\documentclass[10pt]{amsart}

\usepackage[utf8]{inputenc}
\usepackage{color}
\usepackage{mathtools}
\usepackage{amsmath} 
\usepackage{amssymb} 
\usepackage{amsthm}
\usepackage{graphicx}
\usepackage{esint}
\usepackage[colorlinks=true,linkcolor=blue]{hyperref}

\theoremstyle{plain}

\newtheorem{prop}{Proposition}[section]
\newtheorem{lem}{Lemma}[section] \newtheorem{cor}{Corollary}[section]

\newtheorem{defi}{Definition}[section]
\theoremstyle{remark}
\newtheorem{rmk}{Remark}

\newcommand {\R} {\mathbb{R}}  \newcommand {\Q} {\mathbb{Q}}
 \newcommand {\N} {\mathbb{N}}
\newcommand {\C} {\mathbb{C}} 
\newcommand {\p} {\partial}
\newcommand {\dt} {\partial_t}

\newcommand {\drr} {\partial_r}

\newcommand {\va} {\varphi}
\newcommand {\D} {\Delta}

\newcommand {\supp} {\text{supp}}

\DeclareMathOperator{\vol}{vol}
\DeclarePairedDelimiter{\floor}{\lfloor}{\rfloor}

\begin{document}

\title[Quantitative Unique Continuation]{On Some Quantitative Unique Continuation Properties of Fractional Schrödinger Equations: Doubling, Vanishing Order and Nodal Domain Estimates}

\author{Angkana Rüland}
\address{Mathematical Institute of the University of Oxford, Andrew Wiles Building, Radcliffe Observatory Quarter, Woodstock Road, OX2 6GG Oxford, United Kingdom }
\email{ruland@maths.ox.ac.uk}

\begin{abstract}
In this article we determine bounds on the maximal order of vanishing for eigenfunctions of a generalized Dirichlet-to-Neumann map (which is associated with fractional Schrödinger equations) on a compact, smooth Riemannian manifold, $(M,g)$, without boundary. Moreover, with only slight modifications these results generalize to equations with $C^1$ potentials. Here Carleman estimates are a key tool. These yield a quantitative three balls inequality which implies quantitative bulk and boundary doubling estimates and hence leads to the control of the maximal order of vanishing. Using the boundary doubling property, we prove upper bounds on the $\mathcal{H}^{n-1}$-measure of nodal domains of eigenfunctions of the generalized Dirichlet-to-Neumann map on analytic manifolds.
\end{abstract}
\keywords{Quantitative unique continuation, fractional Schrödinger equations, Carleman estimates, eigenfunctions, nodal domains}
\subjclass[2010]{35R11, 35A02, 58J02}

\maketitle
\tableofcontents

\section{Introduction}

Motivated by trying to improve the understanding of the behavior of the fractional Laplacian, this article is dedicated to quantitative unique continuation results for generalized Dirichlet-to-Neumann maps on smooth, $n$-dimensional compact Riemannian manifolds, $(M,g)$, without boundary. Considering solutions of the (degenerate) elliptic Caffarelli-Silvestre extension \cite{CaS} 
\begin{equation}
\label{eq:CSE}
\begin{split}
\p_{n+1}y_{n+1}^{1-2s}\p_{n+1}\bar{w} + y_{n+1}^{1-2s}\D_{g}\bar{w} + y_{n+1}^{1-2s}b\cdot \nabla \bar{w} + y_{n+1}^{1-2s}c \bar{w} &= 0 \mbox{ on } M \times \R_+,\\
\bar{w} &= w \mbox{ on } M \times \{0\}, 
\end{split}
\end{equation} 
where $s\in (0,1)$, $b$ is a smooth vector field and $c$ a smooth function on $M\times \R_+$,
we are interested in the behavior of eigenfunctions of the associated generalized Dirichlet-Neumann-map, i.e. of functions which satisfy $\Lambda(w)= \lambda w$, where
\begin{equation}
\label{eq:GDN}
\begin{split}
\Lambda: H^{s}_g(M) &\rightarrow H^{-s}_g(M),\\
w & \mapsto -\lim\limits_{y_{n+1}\rightarrow 0} y_{n+1}^{1-2s} \p_{n+1} \bar{w}.
\end{split}
\end{equation}
Important examples, in which these Dirichlet-to-Neumann maps play a significant role, include nonlocal powers of the Laplacian in different guises, c.f. \cite{CaS}, \cite{BL14}. \\

Two very natural quantities associated with eigenfunctions of elliptic operators on manifolds are the eigenfunctions' vanishing order and nodal domains. We will study these for our class of non-local operators in the sequel.
Whereas in unbounded domains these eigenfunctions and, more generally, also solutions of $\Lambda(w)= V w$ with, for instance, $C^1$ potentials $V$ are expected to vanish of possibly arbitrarily high order, this is no longer true on compact, smooth Riemannian manifolds without boundary. Here the vanishing order is strongly influenced by the respective potential, e.g. the strength of the eigenvalue, the eigenfunctions are expected to display ``polynomial behavior in $\lambda$''. Similarly, the measure of the nodal domains is controlled by the size of the eigenvalue. \\ 

This behavior of our class of non-local eigenfunctions is in close analogy to the known results on eigenfunctions of the classical Laplacian on compact Riemannian manifolds first established by Donnelly and Fefferman in their seminal paper \cite{DoF}. In \cite{DoF} the authors prove that at any point $p \in M$ of a compact Riemannian manifold without boundary an eigenfunction of the Laplacian with eigenvalue $\lambda$, i.e. a function such that
\begin{align*}
(-\D_g)w = \lambda w \mbox{ on } M,
\end{align*}
satisfies an estimate of the type
\begin{align*}
\left\| w\right\|_{L^2(B_{2r}(p))} \lesssim e^{C(\lambda^{\frac{1}{2}}+1)} \left\| w\right\|_{L^2(B_{r}(p))},
\end{align*}
where all involved constants only depend on the manifold $(M,g)$. Moreover, they give a complete answer to a famous conjecture of Yau \cite{Y82}, which in the two-dimensional case had already previously been pioneered by Brüning \cite{B}, in the setting of compact analytic manifolds without boundary: For this class of manifolds the $\mathcal{H}^{n-1}$ measure of the nodal set of eigenfunctions of the classical Laplacian scales like $\lambda^{\frac{1}{2}}$:
\begin{align*}
\lambda^{\frac{1}{2}} \lesssim \mathcal{H}^{n-1}\left\{x \in M \big| \  w(x)=0  \right\} \lesssim \lambda^{\frac{1}{2}}.
\end{align*}
Both of these results have been subject of intensive research in various settings involving for example rougher manifolds and equations, c.f. \cite{JK}, \cite{KT1}, \cite{GL}, \cite{DoF}, \cite{KukaI}, \cite{Bakri}, \cite{Da}, \cite{HS}, \cite{SZ11}, \cite{CM11}, \cite{HS11} , \cite{HW12}, \cite{S13}.\\

In the sequel we prove an analogous bound on the vanishing order of eigenfunctions of the previously introduced generalized Dirichlet-to-Neumann map as well as similar upper bounds on the size of the nodal domains. In this context, we refer to the order of vanishing of a function $w$ in an $L^2$-sense (and will later show that this can equivalently be reformulated in an $L^{\infty}$-sense):
\begin{defi}
Let $(M,g)$ be a smooth manifold. We say that a function $w \in L^2_g(M)$ vanishes of order $\gamma$ at the a point $p\in M$ if
\begin{align*}
\limsup\limits_{r\rightarrow 0} \frac{\ln(\left\| w \right\|_{L^2_g(B_{r}(p))})}{\ln(r)} = \gamma.
\end{align*}
\end{defi}

\subsection{The Main Results}
While both the qualitative and quantitative unique continuation properties of classical Schrödinger equations have been intensively studied much less is known about the analogous properties of solutions of fractional Schrödinger equations. Only very recently, first results on the weak and strong unique continuation principles for these nonlocal equation appeared \cite{Seo}, \cite{Seo1}, \cite{FF}, \cite{Rue}, \cite{Rue14}. Our results can be regarded as a quantitative version of this line of thought on compact Riemannian manifolds without boundary. \\

The main result of this article states that on compact manifolds (i.e. $M$ being compact) eigenfunctions of the generalized Dirichlet-to-Neumann map (\ref{eq:CSE}), (\ref{eq:GDN}) behave in an analogous way as eigenfunctions of their classical counterpart, the Riemannian Laplacian:

\begin{prop}[Order of vanishing for eigenfunctions]
\label{prop:eigen}
Let $(M,g)$ be a smooth, compact Riemannian manifold without boundary. Let $s\in(0,1)$ and let $w\in H^s_g(M)$ be an eigenfunction of the generalized Dirichlet-to-Neumann map (\ref{eq:CSE}), (\ref{eq:GDN}) with a smooth, bounded vector field $b$ and a smooth and bounded function $c$ on $M \times \R_+$. Then the order of vanishing of $w$ does not exceed
\begin{align*}
C \lambda^{\frac{1}{2s}} + C,
\end{align*}
for some constant $C>0$ which only depends on the manifold $(M,g)$ and the constants $s$, $\left\| b \right\|_{L^{\infty}}$, $\left\| c \right\|_{L^{\infty}}$.
\end{prop}

We stress that in the limit $s\rightarrow 1$ this exactly reproduces the result of Donnelly and Fefferman \cite{DoF}. Moreover, our techniques also transfer to the situation of fractional Schrödinger equations with differentiable potentials:

\begin{prop}[Order of vanishing for $C^1$ potentials]
\label{prop:C1}
Let $(M,g)$ be a smooth, compact Riemannian manifold without boundary. Let $s\in(0,1)$ and let $w\in H^s_g(M)$ be a solution of the generalized Dirichlet-to-Neumann map with  $C^1$ potential $V$, i.e. let $\bar{w}\in H^{1}(y_{n+1}^{1-2s}d\vol_g d y _{n+1}, M \times \R_{+})$ solve (\ref{eq:CSE}) with a smooth, bounded vector field $b$ and a smooth and bounded function $c$ on $M \times \R_+$ and let 
\begin{equation}
\label{eq:GDNP}
\Lambda(w)= Vw \mbox{ on } M.
\end{equation} 
Then the order of vanishing of $w$ does not exceed
\begin{align*}
C \left\| V \right\|_ {C^1}^{\frac{1}{2s}} + C,
\end{align*}
for some constant $C>0$ which only depends on the manifold $(M,g)$ and the constants $s$, $\left\| b \right\|_{L^{\infty}}$, $\left\| c \right\|_{L^{\infty}}$.
\end{prop}

Again, we emphasize that this parallels the behavior of solutions of the classical Laplacian, as for instance described in \cite{Bakri}.\\

Moreover, we show that eigenfunctions of the generalized Dirichlet-to-Neumann map with analytic coefficients on analytic manifolds share further properties of their local analogue by proving an upper bound on the $\mathcal{H}^{n-1}$ Hausdorff-measure of their nodal domains:

\begin{prop}[Nodal domain estimate]
\label{prop:nodaldom}
Let $(M,g)$ be an analytic, $n$-dimensional compact Riemannian manifold without boundary. Let $s\in(0,1)$ and let $w\in H^s_g(M)$ be an eigenfunction of the generalized Dirichlet-to-Neumann operator (\ref{eq:CSE}), (\ref{eq:GDN}) in which all coefficients are real analytic, i.e. the functions $b$, $c$ in (\ref{eq:CSE}) are real analytic functions.
Then, 
\begin{align*}
\mathcal{H}^{n-1}(\{x\in M |  w(x)= 0 \}) \leq  C\lambda^{\frac{1}{2s}} + C
\end{align*}
for some constant $C>0$ which only depends on the manifold $(M,g)$ and the constants $s$, $\left\| b \right\|_{L^{\infty}}$, $\left\| c \right\|_{L^{\infty}}$.
\end{prop}

Again, this property has been derived in the seminal paper of Donnelly and Fefferman \cite{DoF} for the classical Laplacian operator. Recently, the analogue of this problem for the fractional Laplacian has attracted a certain amount of attention with recent articles of Bellova and Lin \cite{BL14} and of Zelditch \cite{Zel14}. In both articles the authors deal with a variant of the half-Laplacian eigenvalue problem, more precisely the Steklov problem. The methods of attacking the problem differ significantly: Bellova and Lin argue via frequency functions while Zelditch makes use of wave groups and microlocal analysis. 
Whereas Zelditch obtains the (conjectured) optimal vanishing order, Bellova and Lin lose polynomial factors in their estimates which thus results in only almost optimal bounds. Simultaneously and independently of the present work the doubling inequalities of Bellova and Lin have been improved to an optimal scaling in the eigenvalue by Zhu \cite{Zhu14}. His results lead to optimal bounds on the vanishing order of eigenfunctions. However, the optimal doubling property is restricted to small $\lambda$-dependent balls and thus necessarily leads to losses if applied to obtain estimates on the size of nodal domains.\\

In the present article we instead make use of Carleman estimates which deal with the resulting boundary contributions in an optimal way. In particular, Proposition \ref{prop:nodaldom} covers the case of the Steklov problem in real analytic domains:

\begin{cor}[Steklov eigenvalues]
\label{cor:Steklov}
Let $\Omega \subset \R^{n+1}$ be a bounded real analytic, $(n+1)$-dimensional manifold with real analytic boundary. Let $w: \Omega \rightarrow \R$ be a solution of
\begin{equation}
\label{eq:Steklov}
\begin{split}
\D w & = 0 \mbox{ in } \Omega,\\
\nu \cdot \nabla w &= \lambda w \mbox{ on } \partial \Omega,
\end{split}
\end{equation} 
where $\nu: \partial \Omega \rightarrow \R^{n+1}$ denotes the outward unit normal field associated with $\Omega$. Then,
\begin{align*}
\mathcal{H}^{n-1}(\{x\in \partial \Omega |  w(x)= 0 \}) \leq  C\lambda +C
\end{align*}
for some constant $C>0$ which only depends on the domain $\Omega$.
\end{cor}

Let us briefly outline the main ideas of the article. As in \cite{DoF}, \cite{Bakri} or \cite{Da}, our main tool consists a (bulk) Carleman estimate which we phrase for a general $C^1$ potential $V$:

\begin{prop}[Variable coefficient Carleman estimate]
\label{prop:Carl}
Let $(M,g)$ be a compact Riemannian manifold without boundary, let $p\in M$. Let $V: M \rightarrow \R$ be a $C^1$ potential with $\left\| V \right\|_{C^1(M)}<\infty$ and $s\in(0,1)$. Assume that $w: M\times \R_+ \rightarrow \R$, $w\in H^{1}(y_{n+1}^{1-2s}d \vol_{g} dy_{n+1}, M\times \R_+)$ with $\supp{(w)} \subset \overline{B_{r_0}^+(p)}\setminus \overline{B_{\epsilon}^+(p)}$, where $B_{r_0}^+(p)$ denotes a sufficiently small geodesic (half-)ball (i.e. $0<\epsilon\ll r_0=r_0(g)$), solves
\begin{align*}
\p_{n+1} y_{n+1}^{1-2s} \p_{n+1} w + y_{n+1}^{1-2s} \D_g w  &= f \mbox{ in } M\times \R_+,\\
-c_s \lim\limits_{y_{n+1}\rightarrow 0} y_{n+1}^{1-2s} \p_{n+1} w &= V w \mbox{ on } M \times \{0\}.
\end{align*}
Let
$$\phi(x)= - \ln(r) + \frac{1}{10}\left( \ln(r)\arctan(\ln(r)) - \frac{1}{2} \ln(1+\ln(r)^2) \right).$$
Then, for $\tau\gtrsim \left\| V \right\|_{C^1(M)}^{\frac{1}{2s}}>0$ we have
\begin{equation}
\label{eq:vCarl}
\begin{split}
&\tau^{\frac{3}{2}} \left\|  e^{\tau \phi} r^{-1}(1+\ln(r)^2)^{-\frac{1}{2}} y_{n+1}^{\frac{1-2s}{2}}  w \right\|_{L^2(B_{r_0}^+(p))} \\
&+  \tau^{\frac{1}{2}}\left\| e^{\tau \phi} (1+\ln(r)^2)^{-\frac{1}{2}}y_{n+1}^{\frac{1-2s}{2}} \nabla w \right\|_{L^2(B_{r_0}^+(p))}
\lesssim  \ \left\|e^{\tau \phi} r  y_{n+1}^{\frac{2s-1} {2}} f \right\|_{L^2(B_{r_0}^+(p))}.
\end{split}
\end{equation}
\end{prop}

However, in contrast to the situation of local Schrödinger operators, a central ingredient in obtaining these estimates consists of controlling the respective boundary contributions on $M\times \{0\}$. Thus, the quantitative unique continuation problem for the fractional Laplacian has the flavor of a quantitative boundary unique continuation result in the spirit of \cite{KEA}, \cite{KW}, \cite{E} (however without the additional difficulty of rough boundaries). We deal with this issue via an appropriate interpolation/ trace estimate.\\
By virtue of three balls and covering arguments, the Carleman estimate (\ref{eq:vCarl}) allows to deduce a doubling property in the bulk with a (for eigenfunctions) optimal dependence on the parameter $\lambda$. 
Building on this, we present two strategies of passing to the respective boundary estimates. In the first and much less quantitative argument we rely on a blow-up technique which then leads to the desired bound on the vanishing order of eigenfunctions. Here, we argue indirectly by a reduction to the weak unique continuation principle. While this strategy is strong enough to yield the optimal bounds on the order of vanishing, it is not quantitative enough to deal with the nodal domain estimates. \\
In order to obtain these, we follow the strategy of Bellova and Lin \cite{BL14} which permits to deduce optimal boundary doubling estimates. Using an analytic extension of eigenfunctions of the fractional Laplacian, we then conclude the estimate on the size of nodal domains by applying an integral geometry estimate \cite{DoF}, \cite{HL10}.\\

Let us finally comment on the organization of the remainder of the article: After briefly motivating our problem and introducing our notational conventions, we present the proof of the crucial Carleman estimate (\ref{eq:vCarl}) in Section 2. Using this as well as elliptic estimates, it then becomes possible to deduce a three balls inequality and a quantified (bulk) doubling estimate. In Section 4, we prove the result on the vanishing rate by reducing the problem to the weak unique continuation setting. Finally, in Sections 5 and 6 we derive an (optimally scaling) boundary doubling estimate as well as bounds on the Hausdorff measure of nodal domains in the setting of analytic manifolds. In the appendix we recall interior and boundary regularity estimates for the generalized harmonic extension operator. 

\subsection{Motivation}
\subsubsection{Heurstics for the Fractional Laplacian on Compact Manifolds}

A motivating example for investigating the generalized Dirichlet-to-Neumann problem (\ref{eq:CSE}), (\ref{eq:GDN}), is given by eigenfunctions of the fractional Laplacian on a compact manifold. In its simplest form this corresponds to studying how fast (nontrivial) solutions, $w\in H^{s}_{g}(M)$, of
\begin{align}
\label{eq:frL}
(-\D_{g})^s w = \lambda w,
\end{align}  
for $\lambda \geq 0$ and $s\in(0,1)$ may vanish. In this context we interpret the fractional Laplacian on a smooth Riemannian manifold $(M,g)$ via its ``Caffarelli-Silvestre extension'' \cite{CaS}. In other words, equation (\ref{eq:frL}) is to be understood as an eigenfunction problem for the ``generalized Dirichlet-to-Neumann'' map of a ``generalized harmonic extension problem'': Considering a function $f\in H^{s}_g(M)$, $s\in(0,1)$, and a solution $\bar{w}\in H^{1}(y_{n+1}^{1-2s}d\vol_g dy_{n+1},M\times \R_+)$ of
\begin{align*}
(\p_{n+1} y_{n+1}^{1-2s} \p_{n+1}  + y_{n+1}^{1-2s} \D_{g}) \bar{w} &= 0 \mbox{ in } M\times \R_{+},\\
\bar{w} &= f \mbox{ on } M\times \R_{+},
\end{align*}
we define $(-\D_g)^s f:=-c_s\lim\limits_{y_{n+1}\rightarrow 0} y_{n+1}^{1-2s} \p_{n+1} \bar{w}$, where $c_s$ is an only $s$-dependent constant which has the precise value $c_s= \pi^{-\frac{n}{2}}2^{2s-1}\frac{\Gamma\left(\frac{n+2s}{2} \right)}{-\Gamma(-s)}$. Thus, for $s\in(0,1)$ the operator we consider is a generalized Dirichlet-to-Neumann map on $M$:
\begin{align*}
(-\D_g)^s: H^{s}_g(M) &\rightarrow H^{-s}_g(M)\\
f &\mapsto (-\D_g)^s f:= -c_s\lim\limits_{y_{n+1}\rightarrow 0} y_{n+1}^{1-2s} \p_{n+1} \bar{w}.
\end{align*}
As a self-adjoint operator with compact inverse, $(-\D_g)^s$ possesses an associated orthonormal basis of eigenfunctions which constitute the solutions of (\ref{eq:frL}).\\ 
In this setting without lower order coefficients $b$ and $c$ (c.f. (\ref{eq:CSE})), the problem can (at least heuristically) be reduced to a eigenvalue equation for a higher order Laplacian operator on the manifold $(M,g)$. This strongly suggests that the powers of $\lambda$ which are given in the previously stated results are optimal (at least for generic manifolds).\\

Let us elaborate on this: Assume that $s\in (0,1)\cap \Q$, $s=\frac{p}{q}$ with $p,q \in \N$ and $q$ minimal. As $(M,g)$ is a compact manifold, $-\D_g$ gives rise to a basis of eigenfunctions, $\varphi_k$, with corresponding eigenvalues $\lambda_k\geq 0$, i.e.
\begin{align*}
-\D_g \varphi_k = \lambda_k \varphi_k .
\end{align*}
Expanding $w(y')= \sum\limits_{k\in \N} \alpha_{k,0} \varphi_k(y')$, we have
\begin{align*}
\bar{w}(y',y_{n+1})= \sum\limits_{k\in \N} \alpha_k(y_{n+1})\varphi_k(y'),
\end{align*} 
with $\alpha_k(0)= \alpha_{k,0}$.
Thus, the eigenvalue problem (\ref{eq:frL}) formally reduces to
\begin{align*}
\sum\limits_{k\in \N}(\p_{n+1} y_{n+1}^{1-2s}\p_{n+1} \alpha_k - y_{n+1}^{1-2s}\lambda_k \alpha_k) \varphi_k &= 0 \mbox{ on } M \times \R_+,\\
- c_s \sum\limits_{k\in \N} \lim\limits_{y_{n+1}\rightarrow 0} y_{n+1}^{1-2s} \alpha_k'(y_{n+1}) \varphi_k(y') &= \lambda \sum\limits_{k\in \N} \alpha_{k,0} \varphi_k(y') \mbox{ on } M \times \{0\}.
\end{align*}
Solving this ODE thus implies the identity
\begin{align*}
\sum\limits_{k\in \N} \lambda_k^{s} \alpha_{k,0} \varphi_k(y') = \lambda \sum\limits_{k\in \N} \alpha_{k,0} \varphi_k(y').
\end{align*}
As a consequence, a $q$-fold application of the fractional Laplacian operator onto $w$ formally results in
\begin{align*}
\sum\limits_{k\in \N} \lambda_k^{p} \alpha_{k,0} \varphi_k(y') = \Lambda^{q}w(y') = \lambda^{q} \sum\limits_{k\in \N} \alpha_{k,0} \varphi_k(y'),
\end{align*}
which formally corresponds to
\begin{align*}
(-\D_{g})^{p} w = \lambda^{q} w \mbox{ on } M.
\end{align*}
Due to the results of Donnelly and Fefferman \cite{DoF} as well as the available results on higher order operators \cite{CK}, \cite{Zhu13}, this exactly corresponds to the vanishing orders given in Propositions \ref{prop:eigen} to \ref{prop:nodaldom}.

\subsubsection{The Steklov Problem}
\label{sec:Steklov}
A second motivation for addressing the eigenvalue problem for the generalized Dirichlet-to-Neumann operator (\ref{eq:CSE}), (\ref{eq:GDN}) stems from the Steklov problem (\ref{eq:Steklov}) which was recently investigated by Bellova and Lin \cite{BL14} as well as by Zelditch \cite{Zel14}. In this section we illustrate how it fits into our framework, c.f. \cite{BL14}. This, thus permits us to recover the results of Bellova \& Lin and Zelditch via our Carleman-inequality based argument.  \\ 

Due to the compactness and regularity of $\Omega$, there exists a radius $0<r<\infty$ such that for each boundary point $y_0\in \partial \Omega$ there exists a neighbourhood $B_{r}(y_0)$ such that
\begin{align*}
B_{r}(x_0)\cap \partial \Omega = \{(y',y_{n+1})\in B_{r}(x_0)| \ y_{n+1}= \Phi(y') \}.
\end{align*} 
Hence, we may flatten the boundary: Setting $F(y)=(y',y_{n+1}+\Phi(y'))$ we define a function $v:B_{r}^+(y'_{0},0) \subset \R^{n+1}_+ \rightarrow \R$,  $v(y)=(w\circ F)(y)=w(y', y_{n+1}+\Phi(y'))$. After multiplying with an appropriate cut-off function, $v$ satisfies an equation which is almost of the form (\ref{eq:CSE}). More precisely, the leading order operator is given by
$\nabla \cdot a \nabla w$ 
where 
\begin{align*}
a^{ij} = \begin{pmatrix} 1 & 0 & 0 &  ...& 0 & -\p_1 \Phi\\
0 & 1 & 0& ... & 0& - \p_2 \Phi\\
& . & .& .&. & \\
-\p_1 \Phi & - \p_2 \Phi & -\p_3 \Phi& ... & -\p_n \Phi  & 1
 \end{pmatrix} + 
\begin{pmatrix}
\nabla' \Phi \otimes \nabla' \Phi & 0\\
0 & 0
\end{pmatrix}
\end{align*}
and the boundary condition turns into
\begin{align*}
\Lambda(v)= \sqrt{1+|\nabla \Phi|^2}\lambda v.
\end{align*}
The lower order coefficients $b,c$ only depend on the geometry of the boundary of $\Omega$. \\

Unlike (\ref{eq:CSE}), the resulting problem is however not posed on a compact manifold $M$ without boundary but on $\R^{n+1}_+$. Yet, this setting suffices in order to deduce the desired Carleman inequality (which is a purely local argument) for $v$ in $\R^{n+1}_+$, c.f. Remark \ref{rmk:Steklov}. Using the non-degeneracy of the coordinate transformation $\Phi$, the Carleman estimates can then be rephrased as an inequality for $w$ in a neighbourhood of $\partial \Omega$ in $\Omega$. Thus, it becomes possible to recover the setting of a compact manifold, which then allows to conclude by exploiting compactness of the domain $\Omega$ and its boundary $\partial \Omega$, e.g. in the three balls argument.

\subsection{Notation}
\label{sec:not}
Let us briefly introduce some of the notation we will use in the sequel:
\begin{itemize}
\item Let $(M,g)$ be a compact manifold of dimension $n$. Then we denote the geodesic distance with respect to a given point $p\in M$ by $r_{p}$. We denote the volume form on $M$ by $d\vol_g$ and the volume form on $M\times \R_+$ by $d\vol_g dy_{n+1}$. We denote the respective $L^2$-based Hilbert spaces on $M$ by $H^{k}_g(M)$; similar notation is used on $M\times \R^+$. If we do not specify the metric in these spaces we always mean the Euclidean one.
\item We write  $B_{r}(p)$ for the geodesic balls with center $p \in M$ and radius $r$ (the distance is measured with respect to the point $p$; here we choose $r$ depending on $g$ sufficiently small such that the geodesic coordinates are uniquely defined). For a point $p \in M$ we use the notation $B_r^+(p)$ in order to refer to the respective half-ball with radius $r$ and center $(p,0)$ in the upper cylinder $M\times \R_+$. If there is no chance of misunderstanding we also suppress the center, $p$, of the ball.
\item Let $0<r_1<r_2$ be given radii and $p\in M$. Then we also use the notation $(r_1,r_2)$ and $(r_1,r_2)^+$ for the annuli $B_{r_2}(p)\setminus B_{r_1}(p) \subset M$ and $B_{r_2}^+(p)\setminus B_{r_1}^+(p) \subset M \times \R_+$ if it is clear around which point the annuli are taken.
\item We use $\nabla$ and $\nabla'$ in order to denote the gradients with respect to $(M\times \R, g\times id)$ and $(M,g)$.
\item In the sequel -- and in particular in the proof of the Carleman estimate -- we will switch between different coordinate systems (and functions). This allows to obtain an as simple as possible structure for the operator we are dealing with. In order to clarify the notation associated with the changes in the coordinate systems, we will always refer
\begin{itemize}
\item to a function in Cartesian variables as $w$,
\item to the associated and scaled function in conformal variables by $u$ and set $u(t,\theta) := e^{\frac{n-2s}{2}t}w(e^{t}\theta)$, 
\item to the function resulting from the $\theta_n$-conjugation by $v$ and define $v(t,\theta):= \theta^{\frac{1-2s}{2}}_n u(t,\theta)$. 
\end{itemize}
\end{itemize}

\section{Proof of the Carleman Estimate}

\begin{proof}[Proof of Proposition \ref{prop:Carl}]
\emph{Step 1: Choice of coordinates. }
We start by introducing Riemannian geodesic coordinates on $(M\times \R_+, g \times id)$ around the point $(p,0)\in M\times \R_+$, which allows to deduce the Carleman estimate along similar lines as in the Euclidean setting.
This change of coordinates straightens out the geodesics passing through the point $(p,0)\in M\times \R_+$. It leads to the new spherical metric $b_{k l}(r,\theta)$ and to a modified operator, which in leading order reads: 
\begin{align*}
&\theta_n^{1-2s}\frac{1}{r^n}\drr(r^{n+1-2s}\drr) + r^{-1-2s}\frac{1}{\sqrt{\det b_{k l}}}\p_{\theta_i} \theta_n^{1-2s} b_{kl}^{-1}(r,\theta)\sqrt{\det{b_{k l}}}\p_{\theta_{j}}.
\end{align*}
Here $\theta_n:= \frac{y_{n+1}}{|y|}$. In computing the new operator, we have suppressed the first order error
\begin{align*}
\theta_n^{1-2s}r^{1-2s}\frac{(\p_{r} \sqrt{\det b})}{\sqrt{\det b}} \p_r,
\end{align*}
which will be justified in Step 4, highlighting its role as a negligible contribution. 
In the sequel, we will also denote the spherical metric $b_{k l}(r,\theta)$ by $b(r,\theta)$ and ignore the first order term involving the $r$-derivatives of $b_{k l}$. It can be treated as a controlled error contribution, which can be absorbed in the positive bulk terms. More precisely, for a sufficiently small radius $r$ (depending only on the manifold $(M,g)$) we have the following error bounds:
\begin{equation}
\label{eq:err1}
\begin{split}
|\p_r(b_{kl}^{-1})| &\leq C(b_{kl}^{-1}),\\
|\p_r(\det b)| &\leq C,\\
0<\lambda \leq \det b &\leq \Lambda. 
\end{split}
\end{equation}
We carry out the change into conformal coordinates, i.e. $r=e^{t}$, which yields $\drr = e^{-t}\dt$. This results in
\begin{align}
\label{eq:conf}
e^{(-1-2s)t}\left[  \theta_n^{1-2s} \dt^2 + \theta_n^{1-2s} \left(n-2s \right)\dt +  \tilde{\nabla}_{S^{n}} \cdot \theta_n^{1-2s} \tilde{\nabla}_{S^n} \right],
\end{align}
where for brevity of notation we used $\tilde{\nabla}_{S^n}$ to denote the spherical gradient with respect to our (non-standard) spherical metric. Thus, the error estimate (\ref{eq:err1}) turns into
\begin{equation}
\label{eq:err2}
\begin{split}
|\p_t(b_{kl}^{-1})| &\leq Ce^{t}(b_{kl}^{-1}),\\
|\p_t (\det b)| &\leq Ce^t,\\
0<\lambda \leq \det b &\leq \Lambda. 
\end{split}
\end{equation}
Conjugating (\ref{eq:conf}) with $e^{-\frac{n-2s}{2}t}$ (which corresponds to setting $w=e^{- \frac{n-2s}{2}t}u$) and multiplying the operator with $e^{(1+2s)t}$, results in
\begin{align*}
\theta_n^{1-2s}\left(\dt^2 - \frac{(n-2s)^2}{4}\right) +  \tilde{\nabla}_{S^{n}}\cdot \theta_n^{1-2s} \tilde{\nabla}_{S^n} .
\end{align*}
Finally, as a last step in the choice of appropriate coordinates, we conjugate the equation with the weight $\theta_n^{\frac{1-2s}{2}}$, i.e. we set $u= \theta_n^{\frac{2s-1}{2}}v$ and multiply the operator with $\theta_n^{-\frac{1-2s}{2}}$. This then results in a symmetric spherical operator in which the spherical and the radial variables are strictly separated:
\begin{align*}
\left(\dt^2 - \frac{(n-2s)^2}{4}\right) +  \theta_n^{\frac{2s-1}{2}}\tilde{\nabla}_{S^{n}}\cdot \theta_n^{1-2s}  \tilde{\nabla}_{S^n}\theta_n^{\frac{2s-1}{2}}.
\end{align*}
In this setting the expression for the boundary values turn into:
\begin{align*}
-\lim\limits_{\theta_{n}\rightarrow 0} \theta_{n}^{1-2s} \nu \cdot \tilde{\nabla}_{S^n}\theta_n^{\frac{2s-1}{2}}v &= \lim\limits_{\theta_{n}\rightarrow 0}\lambda e^{2st}\theta_n^{\frac{2s-1}{2}}v,
\end{align*}
In analogy to the flat case and with a slight abuse of notation, we use the symbol $d \theta$ to denote the volume form of our (non-standard) spherical metric. In the sequel all the integrals will be computed with respect to this volume form.\\

\emph{Step 2: Computing the commutator.}
We conjugate with an only $t$-dependent weight, $\phi$. This leads to the following ``symmetric and antisymmetric'' parts of the operator:
\begin{equation*}
\begin{split}
S & = \dt^2 + \tau^2(\dt \phi)^2 - \frac{(n-2s)^2}{4} + \theta_n^{\frac{2s-1}{2}}\tilde{\nabla}_{S^n}\cdot \theta_n^{1-2s} \tilde{\nabla}_{S^n} \theta_n^{\frac{2s-1}{2}},\\
A & = -2\tau (\dt \phi)\dt - \tau \dt^2 \phi.
\end{split}
\end{equation*}
We point out that these contributions are not precisely symmetric and antisymmetric with respect to our non-standard spherical metric  (in particular boundary contributions have to be taken care of), yet this separation of the full operator into $S$ and $A$ proves to be convenient for the calculations of the pairing $(Su,Au)_{L^2_b(S^n_+\times \R)}$. In these calculations one has to be slightly more careful than in the case of the standard sphere as the metric tensor, and thus the volume element, also depends on the $t$-variable. As a consequence, it is more convenient to calculate the quantities appearing in $(Su,Au)_{L^2_b(S^n_+ \times \R)}$ directly, instead of symmetrizing and antisymmetrizing the respective contributions (however, the calculations are of course motivated by the commutator estimates). 
Indeed, setting $N:=[0,2\pi)\times [-\frac{\pi}{2},\frac{\pi}{2})\times ... \times [0,\frac{\pi}{2})\times \R$ and $\p N := [0,2\pi)\times [-\frac{\pi}{2},\frac{\pi}{2}) \times ... \times \{0\} \times \R$, the most important contributions are given by:
\begin{align*}
&-2\tau(\dt^2 v, \sqrt{\det b}(\dt \phi) \dt v)_{L^2(N )} \\
	& \quad \quad \quad =  \tau(\dt v, \sqrt{\det b} (\dt^2 \phi) \dt v)_{L^2(N)}  + \tau(\dt v, (\dt \sqrt{\det b})\dt \phi \dt v)_{L^2(N)},\\
&-2\tau^3((\dt \phi)^2 v, \dt \phi \dt v \sqrt{\det b})_{L^2(N)}\\
	 & \quad \quad \quad =   6\tau^3 (v, (\dt^2 \phi) (\dt \phi)^2 v \sqrt{\det b})_{L^2(N)}
	 + 2\tau^3(v, (\dt \phi)^3(\dt \sqrt{\det b}) v )_{L^2(N)},\\
&-\tau (\theta_n^{\frac{2s-1}{2}} \p_{\theta_k}\sqrt{\det b} \theta_n^{1-2s} b^{k l} \p_{\theta_l} \theta_n^{\frac{2s-1}{2}} v, 2\dt \phi \dt v - \dt^2 \phi v )_{L^2(N)}\\
	& \quad \quad \quad =  \tau ( \theta_n^{1-2s} \p_{\theta_k} \theta_n^{\frac{2s-1}{2}} v, (\dt \phi)\dt (b^{kl} \sqrt{\det b}) \p_{\theta_l} \theta_n^{\frac{2s-1}{2}} v)_{L^2(N)} \\
	& \quad \quad \quad \quad + \tau  (\sqrt{\det b}\lim\limits_{\theta_n \rightarrow 0} \theta_n^{\frac{2s-1}{2}} \dt e^{2st}Vv, \lim\limits_{\theta_n \rightarrow 0} \theta_n^{\frac{2s-1}{2}} v)_{L^2(\p N)}\\
	  & \quad \quad \quad \quad  -\tau  (\sqrt{\det{b}}\lim\limits_{\theta_n \rightarrow 0}\theta_n^{\frac{2s-1}{2}}\dt v, 				\lim\limits_{\theta_n \rightarrow 0}e^{2st}V\theta_n^{\frac{2s-1}{2}} v)_{L^2(\partial N)},\\
 &-\tau(\dt^2 v, \sqrt{\det b}\dt^2 \phi v)_{L^2(N) } \\
	 & \quad \quad \quad =  \tau (\dt v, \sqrt{\det b} \dt^2 \phi \dt v )_{L^2(N)} + 
	\tau(\dt v, (\dt \sqrt{\det b})\dt^2 \phi v)_{L^2(N)}\\
	 & \quad \quad \quad\quad + \tau(\dt v,  \sqrt{\det b}\dt^3 \phi v)_{L^2(N)}.
\end{align*} 
All the contributions involving derivatives of the metric are treated as error contributions which is justified by the estimates from (\ref{eq:err2}).
Although we have not computed all the contributions appearing in the mixed term $(Au,Su)_{L^2_b(S^n_+ \times \R)}$, the ones from above constitute the main parts; the others are either only error terms in the sense that they involve derivatives of the metric or are of lower order. Hence, -- up to boundary terms -- we conclude that if $\phi$ is sufficiently pseudoconvex, the separation into $S$ and $A$ yields the following positive ``commutator'' contributions:
\begin{align*}
&4\tau^3 \left\| (\dt^2 \phi)^{\frac{1}{2}} \dt \phi v \right\|_{L^2_b( S^n_+ \times \R)}^2
+ 4\tau \left\| (\dt^2 \phi)^{\frac{1}{2}} \dt v \right\|_{L^2_b( S^n_+ \times \R)}^2,
\end{align*}
and an error which consists either of lower order contributions or bulk terms involving $t$ derivatives of the metric $b$, e.g. terms of the form
\begin{equation}
\begin{split}
\label{eq:uariableerr}
\tau^{3} \int\limits_{ S^{n}_+ \times \R} |\dt \phi| | v|^2 |\dt b| d\theta dt, \
\tau \int\limits_{ S^{n}_+ \times \R} |\dt \phi| |\tilde{\nabla} v|^2 |\dt b| d\theta dt.
\end{split}
\end{equation}
We remark that all integrals are calculated with respect to our non-standard spherical metric.
The error terms will be treated separately in Step 4.\\
Making use of the symmetric part of the conjugated operator and an integration by parts argument as in \cite{Rue14}, we infer that we also control spherical gradient contributions. Up to additional error and boundary terms this implies the following Carleman estimate (phrased in terms of the function $u$):
\begin{equation}
\label{vc:Carl}
\begin{split}
&c\tau \left\| \theta_n^{\frac{1-2s}{2}}(\dt^2 \phi)^{\frac{1}{2}}  \dt u \right\|_{L^2_b}^2
+ c\tau \left\| \theta_n^{\frac{1-2s}{2}} (\dt^2 \phi)^{\frac{1}{2}}  \tilde{\nabla}_{S^{n}} u \right\|_{L^2_b}^2 \\
&+ c\tau^3  \left\| \theta_n^{\frac{1-2s}{2}} (\dt^2 \phi)^{\frac{1}{2}}(\dt \phi) u \right\|_{L^2_b}^2
+ \left\| S (\theta^{\frac{1-2s}{2}}_{n} u) \right\|_{L^2_b}^2 + \left\| A(\theta^{\frac{1-2s}{2}}_{n} u) \right\|_{L^2_b}^2 \\
\lesssim & \ \left\| L_{\phi} \theta_n^{\frac{2s-1}{2}} u \right\|_{L^2_b}^2,
\end{split}
\end{equation}
with
\begin{align*}
L_{\phi} = & \ (\dt^2 + \tau^2 (\dt \phi)^2 - \frac{(n-1)^2}{4} - 2\tau(\dt \phi)\dt - \tau \dt^2 \phi) +  \theta_n^{\frac{2s-1}{2}} \tilde{\nabla}_{S^n}\cdot \theta_n^{1-2s} \tilde{\nabla}_{S^n} \theta_n^{\frac{2s-1}{2}}.
\end{align*}
It remains to discuss the boundary and error contributions. \\

\emph{Step 3: Treatment of the boundary contributions.}
We will deal with the boundary contributions via the trace estimate from \cite{Rue}:
\begin{align}
\label{eq:interpol}
\left\| u \right\|_{L^2_b(S^{n-1})}^2 \lesssim \tau^{2-2s}\left\| \theta_n^{\frac{1-2s}{2}} u \right\|_{L^2_b(S^{n}_+ )}^2 +  \tau^{-2s}\left\| \theta_n^{\frac{1-2s}{2}} \nabla u \right\|_{L^2_b( S^{n}_+ )}^2 
\end{align}
This permits us to absorb the boundary contributions in the bulk contributions if $\tau> \lambda^{\frac{1}{2s}}$. Hence, up to the remaining discussion of the error terms, this then establishes the desired Carleman inequality.\\ 
Computing all the boundary contributions which occur in the ``conjugation process'' of step 2 yields:
\begin{align*}
&\int\limits_{\p S^{n}_{+}\times \R}  (\dt^2 \phi)\lim\limits_{\theta_n \rightarrow 0} (\theta^{1-2s}_{n}\nu \cdot \tilde{\nabla}_{S^{n}}\theta^{\frac{2s-1}{2}}_{n}v) \lim\limits_{\theta_n \rightarrow 0}\theta^{\frac{2s-1}{2}}_{n}v d\theta dt \\
= & \ \int\limits_{\p S^{n}_{+}\times \R} (\dt^2 \phi) e^{2st} \lim\limits_{\theta_n \rightarrow 0} \theta_{n}^{2s-1} Vv^2 d\theta dt
\end{align*}
and
\begin{align*}
&4\tau \int\limits_{\p S^{n}_{+}\times \R}(\dt \phi) \lim\limits_{\theta_n \rightarrow 0}\theta^{1-2s}_{n} (\nu \cdot \tilde{\nabla}_{S^{n}}\theta^{\frac{2s-1}{2}}_{n}v)  \lim\limits_{\theta_n \rightarrow 0}\theta^{\frac{2s-1}{2}}_{n}\dt v d\theta dt \\
&+ 2\tau \int\limits_{\p S^{n}_{+}\times \R}(\dt^2 \phi) \lim\limits_{\theta_n \rightarrow 0} \theta^{1-2s}_{n} (\nu \cdot \tilde{\nabla}_{S^{n}}\theta^{\frac{2s-1}{2}}_{n}v) \lim\limits_{\theta_n \rightarrow 0} \theta^{ \frac{2s-1}{2}}_{n} v d\theta dt \\
= & \ 4\tau \int\limits_{\p S^{n}_{+}\times \R} (\dt \phi) e^{2st}\lim\limits_{\theta_n \rightarrow 0} \theta_{n}^{2s-1} V v \dt v d\theta dt \\
& \ +  2\tau \int\limits_{\p S^{n}_{+}\times \R} (\dt^2 \phi)e^{2st} \lim\limits_{\theta_n \rightarrow 0} \theta_{n}^{2s-1}  V v^2 d\theta dt.
\end{align*}
Here the first contribution comes from the gradient estimate while the second originates from the ``symmetrization'' of the operator. Carrying out an integration by parts in the $t$-variable, all the boundary terms reduce to $L^2$ boundary integrals. 
We remark that these terms are all of ``subcritical scaling'' in the $t$ variable and (using the support assumption) can be bounded by
\begin{align*}
\tau \int\limits_{(-\infty,t_0) \times S^{n-1}} \left\| V \right\|_{C^1(M)} e^{2st} u^2 dt d\theta \lesssim & \ \left\| V \right\|_{C^1(M)}  \tau^{3-2s} \left\| \theta_n^{\frac{1-2s}{2}} u \right\|_{L^2_b((-\infty,t_0)  \times S^n_+)}^2\\
& + \left\| V \right\|_{C^1(M)}  \tau^{1-2s} \left\| \theta_n^{\frac{1-2s}{2}} \nabla u \right\|_{L^2_b((-\infty,t_0)  \times S^n_+ )}^2.
\end{align*}
As, due to the choice of $\tau$, we have $\left\| V \right\|_{C^1(M)}  \tau^{-2s} \lesssim 1$, the terms on the right hand side can be controlled by the bulk contributions:
\begin{equation}
\label{eq:int7}
\begin{split} 
\tau \int\limits_{(-\infty,t_0)  \times S^{n-1}} \left\| V \right\|_{C^1(M)}  e^{2st} u^2 dt d\theta \lesssim & \  \tau^{3} \left\| \theta_n^{\frac{1-2s}{2}} u \right\|_{L^2_b((-\infty,t_0)  \times S^n_+)}^2 \\
& \ +  \tau \left\| \theta_n^{\frac{1-2s}{2}} \nabla u \right\|_{L^2_b((-\infty,t_0)  \times S^n_+ )}^2.
\end{split}
\end{equation}

\emph{Step 4: Treatment of the error contributions.} We comment on the first order error terms from the conjugation process.  Instead of including these contributions -- which, in the following, we denote by (Er) -- in the commutator calculation, we treat them as errors:
\begin{align*}
\left\| e^{\tau \phi} L w \right\|_{L^2} = \left\| (S + A + Er)u \right\|_{L^2}
\geq \left\| (S+A)u \right\|_{L^2} - \left\| (Er)u\right\|_{L^2}.
\end{align*}
Due to the smoothness of $g$, the estimates from (\ref{eq:err1}) and (\ref{eq:err2}) and the fact that these terms are only of first order, it is possible to absorb these specific errors into the positive commutator contributions which were deduced in step 2.
\end{proof}

\begin{rmk}
It is possible to lower the regularity of the metric significantly in deducing the Carleman estimate. In fact Lipschitz regularity suffices. As this involved additional technical difficulties we have chosen to work in a smooth setting.
\end{rmk}

\begin{rmk}
\label{rmk:Steklov}
Going through the proof of the Carleman inequality one observes that it remains true for the setting of the transformed Steklov problem as described in \ref{sec:Steklov}. Here one uses geodesic coordinates which are induced by the full operator, in particular it is not necessary to first deduce tangential and then normal geodesic coordinates.
\end{rmk}

\begin{rmk}
\label{rmk:anti}
The Carleman estimate can be slightly sharpened for functions which are compactly supported in annuli around a given point $(p,0)\in M\times \R_+$. 
Assuming that $w$ is supported in $\{\delta \leq |y| \leq R \}$ (here $p=0$ but similar estimates hold for arbitrary $p$) we have
\begin{align*}
&\tau^2 \delta^{-2} \left\| e^{\tau \phi} y_{n+1}^{\frac{1-2s}{2}} w \right\|_{L^2((\delta, c\delta)^+) }^2 \lesssim  \ \left\|e^{\tau \phi}  y_{n+1}^{\frac{2s-1}{2}} |y| \nabla\cdot y^{1-2s}_{n+1}\nabla w \right\|_{L^2((\delta, R)^+) }^2.
\end{align*}
These stronger estimate do not feature the logarithmic loss of the remaining bulk contributions and thus play an important role in proving the doubling property. The estimates are obtained from the antisymmetric part of the operator via an application of Poincar\'e's or Hardy's inequality, c.f. \cite{Rue14} for a detailed derivation.
\end{rmk}

\section[Doubling]{Consequences of the Carleman Estimate \ref{prop:Carl}: Three Balls Inequality and Doubling}
\subsection{Three Balls Inequalities}

In this section we present some of the consequences of the Carleman estimate (\ref{eq:vCarl}). In this context, the main aim is to deduce doubling properties for solutions of the Caffarelli-Silvestre extension (\ref{eq:CSE}) in balls located at the boundary of the domain. To this end, we argue via a combination of three balls and covering arguments which exploit the compactness of our manifold. \\

We begin by proving a three balls theorem in the bulk. Here the main difficulty lies in obtaining sufficiently good bounds on the boundary contributions. In order to avoid losses in the doubling property which may originate from working on too small balls, we do not only rely on the Hardy-trace estimate, as this would enforce a restriction onto balls of the size $\sim \left\|V \right\|_{L^2}^{- \frac{1}{2s}}$. Such a restriction would in fact lead to a similar constraint as the one Bellova and Lin \cite{BL14} have to face and which is also ultimately responsible for the losses in their doubling estimate. Thus, instead, we employ the weighted interpolation inequality from \cite{Rue}. This yields sufficient freedom to obtain the optimal ``bulk doubling property''.

\begin{prop}[Three balls inequality]
\label{prop:3ba}
Let $(M,g)$ be a smooth, compact Riemannian manifold without boundary. Let $V:M\rightarrow \R$ be differentiable with $\left\| V\right\|_{C^1(M)}<\infty$. Assume that $s\in(0,1)$ and $w: M\times \R_+ \rightarrow \R$, $w\in H^{1}(y_{n+1}^{1-2s}d\vol_g d y_{n+1}, M\times \R_+)$, solves
\begin{align*}
\p_{n+1} y_{n+1}^{1-2s} \p_{n+1} w + y_{n+1}^{1-2s} \D_g w + y_{n+1}^{1-2s} b \cdot \nabla w + y_{n+1}^{1-2s}c w &= 0 \mbox{ in } M\times \R_+,\\
-c_s\lim\limits_{y_{n+1}\rightarrow 0} y_{n+1}^{1-2s} \p_{n+1} w &= V w \mbox{ on } M \times \{0\},
\end{align*}
where $b$ is a bounded vector field and $c$ is bounded function on $M\times \R_{+}$. 
Then there exists a constant $C>0$, a radius $r_0>0$ (which both only depend on $(M,g)$ and the constants $s$, $\left\| b \right\|_{L^{\infty}}$, $\left\| c \right\|_{L^{\infty}}$) and a parameter $\alpha\in (0,1)$ such that for all radii $r$ with $0<r\leq r_0$ and any point $p\in M$ it holds
\begin{align*}
\left\| y_{n+1}^{\frac{1-2s}{2}} w \right\|_{L^2(B_{r}^+(p))} 
\lesssim & \ e^{C(\left\| V \right\|_{C^1(M)}^{\frac{1}{2s}}+1)} \left\| y_{n+1}^{\frac{1-2s}{2}} w \right\|_{L^2(B_{2 r}^+(p))}^{\alpha}  \left\| y_{n+1}^{\frac{1-2s}{2}} w \right\|_{L^2(B_{\frac{r}{2}}^+(p))}^{1- \alpha}.
\end{align*}
\end{prop}

\begin{rmk}
\label{rmk:ratios}
\begin{itemize}
\item We point out that the ratios of the balls involved in the three balls inequality are flexible: Instead of using the radii $r$, $2r$ and $\frac{r}{2}$ for some $0<r\leq r_0(g)$, it would have been possible to use any other triple $(r_1,r_2,r_3)$ with $0<r_1\leq r_2 \leq r_3\leq r_0(t)$ which would then have yielded the estimate
\begin{align*}
\left\| y_{n+1}^{\frac{1-2s}{2}} w \right\|_{L^2(B_{r_2}^+)} 
\lesssim & \ e^{C(\left\| V \right\|_{C^1(M)}^{\frac{1}{2s}}+ 1)} \left\| y_{n+1}^{\frac{1-2s}{2}} w \right\|_{L^2(B_{r_3}^+)} ^{\alpha}  \left\| y_{n+1}^{\frac{1-2s}{2}} w \right\|_{L^2(B_{r_1}^+)} ^{1- \alpha},
\end{align*}
where the respective constant depends on the choice of $(r_1,r_2,r_3)$.
\item For $V=\lambda$ the three balls inequality reads
\begin{align*}
\left\| y_{n+1}^{\frac{1-2s}{2}} w \right\|_{L^2(B_{r}^+(p))} 
\lesssim & \ e^{C(\lambda^{\frac{1}{2s}}+1)} \left\| y_{n+1}^{\frac{1-2s}{2}} w \right\|_{L^2(B_{2 r}^+(p))}^{\alpha}  \left\| y_{n+1}^{\frac{1-2s}{2}} w \right\|_{L^2(B_{\frac{r}{2}}^+(p))}^{1- \alpha}.
\end{align*}
\end{itemize}
\end{rmk}

\begin{proof}
We argue by applying the Carleman inequality (\ref{eq:vCarl}) in combination with an elliptic/Caccioppoli estimate in an appropriately chosen annulus around $p$. We recall Caccioppoli's estimate:
Let $\psi$ be a radial cut-off function supported in an annulus given by $0<\frac{\tilde{r}_0}{2} \leq|y| \leq  2r_{1}<\infty$. Then,
\begin{equation}
\begin{split}
\label{eq:ellreg}
\left\| y_{n+1}^{\frac{1-2s}{2}} \nabla (w\psi) \right\|_{L^2((\frac{\tilde{r}_{0}}{2},2r_{1})^+)}^2
\lesssim & \ \tilde{r}_{0}^{-2}\left\| y_{n+1}^{\frac{1-2s}{2}}  w \right\|_{L^2((\frac{\tilde{r}_{0}}{2},2r_{1})^+)}^2 \\& + \int\limits_{ (\frac{\tilde{r}_{0}}{2},2r_{1})} \psi w \lim\limits_{y_{n+1}\rightarrow 0} y^{1-2s}_{n+1} \p_{n+1}( \psi w) dy\\
\lesssim & \ \tilde{r}_{0}^{-2}\left\| y_{n+1}^{\frac{1-2s}{2}}  w \right\|_{L^2((\frac{\tilde{r}_{0}}{2},2r_{1})^+)}^2  \\
& \ + \left\| V \right\|_{C^1(M)}\int\limits_{ (\frac{\tilde{r}_{0}}{2},2r_{1})} \psi^2 w^2  dy,
\end{split}
\end{equation}
with $0<\tilde{r}_{0}<r_{1}<\infty$. Here we used that normal derivatives which fall on $\psi$ vanish due to the radial dependence of $\psi$. For notational convenience, we in the sequel assume $\tilde{r}_0 \leq 1$ and $\left\| V\right\|_{C^1} \geq 1$. Aiming at a ``bulk estimate'', we absorb the boundary contributions on the right hand side of (\ref{eq:ellreg}) by exploiting the interpolation estimate (\ref{eq:interpol}) from \cite{Rue}.
Hence,
\begin{align*}
\left\| V \right\|_{C^1(M)} \int\limits_{ (\frac{\tilde{r}_{0}}{2},2r_{1}) } \psi^2 w^2  dy
 \lesssim & \ \left\| V \right\|_{C^1(M)} \tau^{2-2s} \left\| y_{n+1}^{\frac{1-2s}{2}} \psi w \right\|_{L^2((\frac{\tilde{r}_{0}}{2},2r_{1})^+)}^2 \\
& + \left\| V \right\|_{C^1(M)} \tau^{-2s} \left\| y_{n+1}^{\frac{1-2s}{2}} \nabla (\psi w) \right\|_{L^2((\frac{\tilde{r}_{0}}{2},2r_{1})^+)}^2.
\end{align*}
Choosing $\tau\geq C\left\| V \right\|_{C^1(M)}^{\frac{1}{2s}}$, for an appropriately large constant $C>0$, the gradient contribution can be absorbed in the left hand side of (\ref{eq:ellreg}). Thus, the Caccioppoli estimate reduces to 
\begin{equation}
\begin{split}
\label{eq:ellreg1}
\left\| y_{n+1}^{\frac{1-2s}{2}} \nabla (w\psi) \right\|_{L^2((\frac{\tilde{r}_{0}}{2},2r_{1})^+)}^2
\lesssim & \ \left\| V \right\|_{C^1(M)} \tau^{2-2s} \tilde{r}_{0}^{-2}\left\| y_{n+1}^{\frac{1-2s}{2}}  w \right\|_{L^2((\frac{\tilde{r}_{0}}{2},2r_{1})^+)}^2,  
\end{split}
\end{equation}
where $\tau\geq \left\| V \right\|_{C^1(M)}^{\frac{1}{2s}}$ and $0<r_0<r_1<\infty$.\\
Keeping the elliptic estimate in mind, we apply the Carleman inequality from Proposition \ref{prop:Carl}: Let $\eta$ be a radial cut-off function, which is equal to one on the annulus $|y|\in (\frac{r}{3}, \frac{3r}{2})$ and vanishes outside of the annulus $|y|\in (\frac{r}{4}, \frac{5r}{3})$. 
Inserting $\eta w$ into the Carleman estimate, using the elliptic estimate (\ref{eq:ellreg1}) as well as the explicit form of the boundary contribution, we obtain 
\begin{align*}
\tau \left\| e^{\tau \phi} y^{\frac{1-2s}{2}}_{n+1} w \right\|_{L^2((\frac{r}{3}, \frac{3r}{2})^+)}^2 
\lesssim & \  \left\| V \right\|_{C^1(M)} \tau^{2-2s} \left( e^{2 \tau \phi(r/4)} \left\|  y^{\frac{1-2s}{2}}_{n+1} w \right\|_{L^2((\frac{r}{4}, \frac{r}{3})^+)}^2 \right. \\
& \  \left. +  \ e^{2\tau \phi(3r/2)} \left\| y^{\frac{1-2s}{2}}_{n+1} w \right\|_{L^2((\frac{3r}{2},\frac{5r}{3})^+)}^2\right).
\end{align*} 
Using the monotonicity properties of $\phi$, we infer
\begin{align*} 
\left\|  y^{\frac{1-2s}{2}}_{n+1} w \right\|_{L^2((\frac{r}{3}, r)^+)}^2 
\lesssim & \  \left\| V \right\|_{C^1(M)} \tau^{2-2s} \left( e^{2\tau(\phi(r/4)-\phi(r))}\left\|  y^{\frac{1-2s}{2}}_{n+1} w \right\|_{L^2((\frac{r}{4}, \frac{r}{3})^+)}^2 \right.\\
& \ \left. + \ e^{2\tau(\phi(3r/2)- \phi(r))} \left\| y^{\frac{1-2s}{2}}_{n+1} w \right\|_{L^2((\frac{3r}{2},\frac{5r}{3})^+)}^2 \right).
\end{align*} 
Adding $ \left\|  y^{\frac{1-2s}{2}}_{n+1} w \right\|_{L^2(0,\frac{r}{3})^+}^2 $ to both sides (and absorbing it in the first contribution of the right hand side), setting $A:= 2(\phi(r/4)-\phi(r))$ as well as $B:= 2(\phi(r)-\phi(3r/2))$ (and noticing that these are can be controlled independently of $r$ as a consequence of the choice of our weight function) results in
\begin{equation}
\label{eq:aux}
\begin{split}
 \left\|  y^{\frac{1-2s}{2}}_{n+1} w \right\|_{L^2(B_{ r}^+)}^2
\lesssim & \ \left\| V \right\|_{C^1(M)} \tau^{2-2s} \left( e^{\tau A} \left\|  y^{\frac{1-2s}{2}}_{n+1} w \right\|_{L^2(B_{r/2}^+)}^2 \right.\\
& \left. +  \ e^{-\tau B} \left\| y^{\frac{1-2s}{2}}_{n+1} w \right\|_{L^2(B_{2r}^+)}^2 \right).
\end{split}
\end{equation}
We now choose the parameter $\tau$ such that the second contribution on the right hand side of (\ref{eq:aux}) can be absorbed in the left hand side of (\ref{eq:aux}). This can, for instance, by achieved by defining
\begin{align*}
\tau \geq  -\frac{1}{B} \ln \left( \frac{1}{\left\| V \right\|_{C^1(M)}} \frac{ \left\| y^{\frac{1-2s}{2}}_{n+1} w \right\|_{L^2(B_{r}^+)} }{  \left\|y^{\frac{1-2s}{2}}_{n+1} w  \right\|_{L^2(B_{2r}^+)} } \right). 
\end{align*}
As the application of the Carleman estimate and of (\ref{eq:ellreg1}) requires $\tau \geq \left\| V \right\|_{C^1(M)}^{\frac{1}{2s}}$, we finally set
\begin{align*}
\tau =  -\frac{1}{B} \ln \left( \frac{1}{\left\| V \right\|_{C^1(M)}} \frac{ \left\| y^{\frac{1-2s}{2}}_{n+1} w \right\|_{L^2(B_{r}^+)} }{  \left\|y^{\frac{1-2s}{2}}_{n+1} w  \right\|_{L^2(B_{2r}^+)} } \right) + C\left\| V \right\|_{C^1(M)}^{\frac{1}{2s}}. 
\end{align*}
For $\alpha = \frac{A}{A+B}$ this then yields the desired estimate:
\begin{align*}
\left\| y_{n+1}^{\frac{1-2s}{2}} w \right\|_{L^2(B_{r}^+)}
\leq & \ e^{C(\left\| V \right\|_{C^1(M)}^{\frac{1}{2s}}+1)}\left\| y_{n+1}^{\frac{1-2s}{2}} w \right\|_{L^2(B_{2r}^+)} ^{\alpha} \left\| y_{n+1}^{\frac{1-2s}{2}} w \right\|_{L^2(B_{r/2}^+)}^{1-\alpha}.
\end{align*}
\end{proof}

Using the three balls inequality, it is possible to obtain a ``global'' estimate. This is of central importance in deducing the desired (bulk and boundary) doubling property with a constant \emph{independent} of $w$. 

\begin{cor}[Overlapping balls argument, I]
\label{cor:obaI}
Let $(M,g)$ be a smooth, compact Riemannian manifold without boundary. Let $V:M\rightarrow \R$ be differentiable with $\left\| V\right\|_{C^1(M)}<\infty$ and let $r_0>0$ be the radius from the previous Lemma. Assume that $s\in(0,1)$ and $w: M\times \R_+ \rightarrow \R$, $w\in H^{1}(y_{n+1}^{1-2s}d\vol_g d y_{n+1}, M\times \R_+)$, solves
\begin{align*}
\p_{n+1} y_{n+1}^{1-2s} \p_{n+1} w + y_{n+1}^{1-2s} \D_g w + y_{n+1}^{1-2s} b \cdot \nabla w + y_{n+1}^{1-2s}c w &= 0 \mbox{ in } M\times \R_+,\\
-c_s \lim\limits_{y_{n+1}\rightarrow 0} y_{n+1}^{1-2s} \p_{n+1} w &= V w \mbox{ on } M \times \{0\},
\end{align*}
where $b$ is a bounded vector field and $c$ is bounded function on $M\times \R_{+}$. 
Then for radii with $0<r\lesssim r_0$ there exists a constant $C_r>0$ (which only depends on $r$, $(M,g)$ and the constants $s$, $\left\| b \right\|_{L^{\infty}}$, $\left\| c \right\|_{L^{\infty}}$) such that 
\begin{align*}
\left\| y_{n+1}^{\frac{1-2s}{2}} w \right\|_{L^2(B_{r}^+)}
 \geq  \ e^{-C_r (\left\| V \right\|_{C^1(M)}^{\frac{1}{2s}}+1)} \left\| y_{n+1}^{\frac{1-2s}{2}} w \right\|_{L^2(M \times [0,4r])} .
\end{align*}
\end{cor}

\begin{rmk}
As in the three balls inequality the exact ratios of of the balls and stripes in the estimate do not matter: Up to changing the constant in the estimate, it would have been possible to prove any inequality of the form
\begin{align*}
\left\| y_{n+1}^{\frac{1-2s}{2}} w \right\|_{L^2(B_{r}^+)} 
 \geq  \ e^{-C_{r,c} (\left\| V \right\|_{C^1(M)}^{\frac{1}{2s}}+1)} \left\| y_{n+1}^{\frac{1-2s}{2}} w \right\|_{L^2(M \times [0,cr])},
\end{align*}
with $c>1$.
\end{rmk}

\begin{proof}
Rescaling we may without loss of generality assume that 
\begin{align*}
\left\| y_{n+1}^{\frac{1-2s}{2}} w \right\|_{L^2(M \times [0,4r])} = 1.
\end{align*}
Let $\bar{y}:=(\bar{y}',0)\in M\times \{0\}$ be such that
\begin{align*}
&\left\| y_{n+1}^{\frac{1-2s}{2}} w \right\|_{L^2(B_{r}^+(\bar{y}))} 
= \ \max\limits_{y\in M\times \{0\}}\left\| y_{n+1}^{\frac{1-2s}{2}} w \right\|_{L^2(B_{r}^+(y))} 
=: \tilde{C}_r
\end{align*}
Let $y_0\in M\times \{0\}$ be arbitrary but fixed. Due to the compactness of $M$ for any pair $(r_1,r)$, $0<r_1\leq r\leq r_0(g)$ there exists a finite number of points $y_{1},...,y_{m-1}\in M\times \{0\}$ connecting $y_0$ and $y_m:= \bar{y}$ such that $B_{r_1}(y_{i+1})\subset B_{r}(y_i)$ for $i \in \{0,...,m-1\}$. Hence, the three balls inequality with appropriately adapted, possibly dimension-dependent radii $0< r\leq r_2$ (c.f. Remark \ref{rmk:ratios}) implies
\begin{align*}
\left\| y_{n+1}^{\frac{1-2s}{2}} w \right\|_{L^2(B_{r}^+(y_0))} 
\gtrsim & \ e^{-C (\left\| V\right\|_{C^1(M)}^{\frac{1}{2s}}+1)} \left\| y_{n+1}^{\frac{1-2s}{2}} w \right\|_{L^2(B_{r_2}^+(y_0))} \\
\gtrsim & \ e^{-C (\left\| V\right\|_{C^1(M)}^{\frac{1}{2s}}+1)} \left\| y_{n+1}^{\frac{1-2s}{2}} w \right\|_{L^2(B_{r}^+(y_1))} 
\end{align*}
Iterating this procedure finally yields
\begin{align}
\label{eq:lb}
 \left\| y_{n+1}^{\frac{1-2s}{2}} w \right\|_{L^2(B_{r}^+(\bar{y}))} 
\gtrsim  \ e^{-C (\left\| V\right\|_{C^1(M)}^{\frac{1}{2s}}+1)} \tilde{C}_{r}.
\end{align}
Thus, undoing the rescaling leads to the desired estimate. 
\end{proof}

As a consequence of the previous overlapping balls argument, it is possible to deduce an analogous result for annuli. This will become relevant when proving the doubling property.

\begin{cor}[Overlapping balls argument, II]
\label{cor:obaII}
Let $(M,g)$ be a smooth, compact Riemannian manifold without boundary.  Let $V:M\rightarrow \R$ be differentiable with $\left\| V\right\|_{C^1(M)}<\infty$ and let $r_0>0$ be the radius from the previous lemma. Assume that $s\in(0,1)$ and $w: M\times \R_+ \rightarrow \R$, $w\in H^{1}(y_{n+1}^{1-2s}d\vol_g d y_{n+1}, M\times \R_+)$, solves
\begin{align*}
\p_{n+1} y_{n+1}^{1-2s} \p_{n+1} w + y_{n+1}^{1-2s} \D_g w + y_{n+1}^{1-2s} b \cdot \nabla w + y_{n+1}^{1-2s}c w &= 0 \mbox{ in } M\times \R_+,\\
-c_s \lim\limits_{y_{n+1}\rightarrow 0} y_{n+1}^{1-2s} \p_{n+1} w &= V w \mbox{ on } M \times \{0\},
\end{align*}
where $b$ is a bounded vector field and $c$ is bounded function on $M\times \R_{+}$. 
Then for any radius $0<r<r_0$ there exists a constant $C_r>0$ (which only depends on $r$, $(M,g)$ and the constants $s$, $\left\| b \right\|_{L^{\infty}}$, $\left\| c \right\|_{L^{\infty}}$)  such that
\begin{align*}
\left\| y_{n+1}^{\frac{1-2s}{2}} w \right\|_{L^2((\frac{r}{2},r)^+)} 
 \geq & \ e^{-C_r (\left\| V\right\|_{C^1(M)}^{\frac{1}{2s}}+1)}  \left\| y_{n+1}^{\frac{1-2s}{2}} w \right\|_{L^2(M \times [0,4r])}.
\end{align*}
\end{cor}

\begin{proof}
The claim follows by noticing that the annulus under consideration contains a ball of (slightly smaller but) comparable size for which we may apply the previous overlapping balls argument (with slightly modified -- though comparable -- radii).
\end{proof}

\subsection{A Quantitative Doubling Estimate in the Bulk}
With the previous auxiliary results it now becomes possible to prove a ``bulk'' doubling estimate:

\begin{prop}[Doubling]
\label{prop:doubl}
Let $(M,g)$ be a smooth, compact Riemannian manifold without boundary. Let $V:M\rightarrow \R$ be differentiable with $\left\| V\right\|_{C^1(M)}<\infty$ and let $r_0>0$ be the radius from the previous lemmas. Assume that $s\in(0,1)$ and that $w: M\times \R_+ \rightarrow \R$, $w\in H^{1}(y_{n+1}^{1-2s}d\vol_g d y_{n+1}, M\times \R_+)$, solves
\begin{align*}
\p_{n+1} y_{n+1}^{1-2s} \p_{n+1} w + y_{n+1}^{1-2s} \D_g w + y_{n+1}^{1-2s} b \cdot \nabla w + y_{n+1}^{1-2s}c w&= 0 \mbox{ in } M\times \R_+,\\
-c_s \lim\limits_{y_{n+1}\rightarrow 0} y_{n+1}^{1-2s} \p_{n+1} w &= V w \mbox{ on } M \times \{0\},
\end{align*}
where $b$ is a bounded vector field and $c$ is bounded function on $M\times \R_{+}$. 
Then there exists a constant $C>0$ (which only depends on $(M,g)$ and the constants $s$, $\left\| b \right\|_{L^{\infty}}$, $\left\| c \right\|_{L^{\infty}}$) such that for any radius, $0<r\leq r_0$, and any point $p\in M$ it holds
\begin{align*}
\left\| y_{n+1}^{\frac{1-2s}{2}} w \right\|_{L^2(B_{2r}^+(p))}
\lesssim e^{C(\left\| V\right\|_{C^1(M)}^{\frac{1}{2s}}+1)}\left\| y_{n+1}^{\frac{1-2s}{2}} w \right\|_{L^2(B_{r}^+(p))}. 
\end{align*}
\end{prop}

\begin{rmk}
We stress that the radius $r_0>0$ in which our boundary doubling inequality is valid is \emph{independent} of the eigenvalue $\lambda$ and the potential $V$. This remains true in the application to the Steklov problem since we do not transfer the parameter $\lambda$ from the boundary onto the interior coefficients of the equation but keep as a quantity on the boundary instead. As illustrated in the previous sections it is dealt with by an application of the interpolation inequality boundary (\ref{eq:int7}).
\end{rmk}

\begin{proof}[Proof of the Proposition \ref{prop:doubl}]
In order to bound the gradient contributions which arise in the application of the Carleman inequality (\ref{eq:vCarl}), we again rely on a combination of the Carleman estimate from Proposition \ref{prop:Carl} and the elliptic regularity estimate (\ref{eq:ellreg}): Let $\eta$ be a radial cut-off function, which is equal to one on the annulus $|y|\in (\delta, R/2)^+$ and vanishes outside of the annulus $|y|\in (\delta/2, R)^+$. 
Inserting $\eta w$ into the Carleman estimate (in combination with Remark \ref{rmk:anti}), using the elliptic estimate (\ref{eq:ellreg1}) (which in particular necessitates $\tau \geq \left\| V\right\|_{C^1}^{\frac{1}{2s}}$), we obtain 
\begin{align*}
&\delta^{-2} \tau \left\| e^{\tau \phi} y^{\frac{1-2s}{2}}_{n+1} w \right\|_{L^2((\delta, 3\delta)^+)}^2 
+ \tau^2 R^{-2} \left\| e^{\tau \phi} (1+\ln(|y|)^{2})^{-\frac{1}{2}}y^{\frac{1-2s}{2}}_{n+1} w \right\|_{L^2((\frac{R}{8}, \frac{R}{4})^+)}^2 \\
\lesssim & \ \left\| V\right\|_{C^1(M)} \tau^{2-2s} \left(\delta^{-2}e^{2 \tau \phi(\delta/2)} \left\|  y^{\frac{1-2s}{2}}_{n+1} w \right\|_{L^2((\frac{\delta}{2}, \frac{3\delta}{2})^+)}^2 \right. \\
&\left. + \ R^{-2}e^{2\tau \phi(R/2)} \left\| y^{\frac{1-2s}{2}}_{n+1} w \right\|_{L^2((\frac{R}{2},2R)^+)}^2 \right).
\end{align*} 
Setting $R\sim 1$ (in dependence of $M$ and $r_0(g)$), we further deduce: 
\begin{align*}
& e^{\tau \phi(3 \delta )}\left\|  y^{\frac{1-2s}{2}}_{n+1} w \right\|_{L^2(B_{3\delta}^+)}^2
+ e^{\tau\phi(R/4)} \delta^2 \tau^2 \left\|   y^{\frac{1-2s}{2}}_{n+1}w \right\|_{L^2((\frac{R}{8},\frac{R}{4})^+)}^2\\
\lesssim & \  \left\| V\right\|_{C^1(M)} \tau^{2-2s} \left( \delta^2 e^{\tau \phi(R/2)} \left\|   y^{\frac{1-2s}{2}}_{n+1} w\right\|_{L^2(B_{2R}^+)}^2 
+ e^{\tau \phi(\frac{\delta}{2})} \left\|   y^{\frac{1-2s}{2}}_{n+1} w\right\|_{L^2(B_{\frac{3\delta}{ 2}}^+)}^2 \right).
\end{align*}
We choose $\tau>0$ such that the contribution $\left\| V\right\|_{C^1(M)} \tau^{2-2s} e^{\tau \phi(R/2)} \left\|   y^{\frac{1-2s}{2}}_{n+1} w\right\|_{L^2(B_{2R}^+)}^2$ on the right hand side can be absorbed in the term $e^{\tau\phi(R/4)}\tau^2 \left\|   y^{\frac{1-2s}{2}}_{n+1}w \right\|_{L^2((\frac{R}{8},\frac{R}{4})^+)}^2 $ on the left hand side. A possible choice of such a $\tau$, for example, is
\begin{align*}
\tau \sim \frac{1}{\phi(R/2)- \phi(R/4)} \ln \left( \frac{1}{\left\| V\right\|_{C^1(M)}^{2s}} \frac{ \left\| y^{\frac{1-2s}{2}}_{n+1} w \right\|_{L^2((\frac{R}{8},\frac{R}{4})^+)} }{  \left\|y^{\frac{1-2s}{2}}_{n+1} w  \right\|_{L^2(B_{2R}^+)} } \right) + C\left\| V\right\|_{C^1(M)}^{\frac{1}{2s}}.
\end{align*}
This then yields
\begin{align*}
\left\| y_{n+1}^{\frac{1-2s}{2}} w \right\|_{L^2(B_{3\delta}^+)} 
\lesssim 
e^{C(1+\left\| V\right\|_{C^1(M)}^{\frac{1}{2s}})} \left(  \frac{  \left\|y^{\frac{1-2s}{2}}_{n+1} w  \right\|_{L^2(B_{2R}^+)}}{ \left\| y^{\frac{1-2s}{2}}_{n+1} w \right\|_{L^2((\frac{R}{8},\frac{R}{4})^+)}} \right)^A
 \left\| y_{n+1}^{\frac{1-2s}{2}} w \right\|_{L^2(B_{3\delta/2}^+)},
\end{align*}
for an absolute constant $A$. Via the three balls inequality and Corollaries \ref{cor:obaI}, \ref{cor:obaII} 
it is possible to further control the quotient by an absolute constant (which depends on $R$ and through $r_0$ also on $M$): 
\begin{align*}
 \frac{  \left\|y^{\frac{1-2s}{2}}_{n+1} w  \right\|_{L^2(B_{2R}^+)}  } { \left\| y^{\frac{1-2s}{2}}_{n+1} w \right\|_{L^2((\frac{R}{8},\frac{R}{4})^+)}  } \leq C e^{C_{R} (\left\| V\right\|_{C^1(M)}^{\frac{1}{2s}}+1)}.
\end{align*} 
Setting $\delta:= \frac{2 r}{3}$, then yields the doubling inequality for small balls with radius $0<r\leq r_0$. 
For large radii $r_0\leq r\leq  C(M)$ we argue via Corollary \ref{cor:obaI}: 
\begin{align*}
\left\| y_{n+1}^{\frac{1-2s}{2}} w \right\|_{L^2(B_r^+)} 
& \geq \left\| y_{n+1}^{\frac{1-2s}{2}} w \right\|_{L^2(B_{r_0}^+)} \\
&\geq e^{-C_{r_0}(\left\| V\right\|_{C^1(M)}^{2s}+1)} \left\| y_{n+1}^{\frac{1-2s}{2}} w \right\|_{L^2(M\times [0,4r])}\\
& \geq e^{-C_{r_0}(\left\| V\right\|_{C^1(M)}^{2s}+1)} \left\| y_{n+1}^{\frac{1-2s}{2}} w \right\|_{L^2(B_{2r}^+)}. 
\end{align*}
\end{proof}

\section{Lower Bounds on the Order of Vanishing}
\label{sec:van}

\subsection{A First Doubling Estimate on the Boundary}
In this section we give a first argument for the order of vanishing, Propositions \ref{prop:eigen} and \ref{prop:C1}. To this end, we transfer the ``bulk doubling property'' into a growth estimate at the boundary:

\begin{cor}
\label{cor:vani}
Let $(M,g)$ be a compact, smooth Riemannian manifold without boundary. Assume that $s\in(0,1)$ and that $w: M \rightarrow \R$, $w\in H^{s}_g(M)$, is an eigenfunction of the generalized Dirichlet-to-Neumann map (\ref{eq:CSE}), (\ref{eq:GDN}).
Let $R\sim r_0(g)$ be fixed (here $r_0(g)$ is the radius from the Carleman estimate (\ref{eq:vCarl})). Then there exist constants $C_1=C_1(\lambda, r_0)$, $C_2$ such that for all sufficiently small radii $r$ with $0<r\ll R$ it holds
\begin{align*}
\left\| w \right\|_{L^2(B_{R})} \leq  C_1(\lambda, r_0) \left( \frac{R}{r}\right)^{C_2 (\lambda^{\frac{1}{2s}}+ 1)} \left\| w \right\|_{L^2(B_r)}.  
\end{align*}
The constant $C_2$ only depends on the manifold $(M,g)$ and the constants $s$, $\left\| b \right\|_{L^{\infty}}$, $\left\| c \right\|_{L^{\infty}}$.
\end{cor}

We aim at deducing this estimate from the bulk doubling estimate by proving that the bulk mass can be controlled from above and below by the boundary mass. We first estimate the $L^2$-boundary norm by the weighted bulk norm. This is a consequence of the interpolation inequality (\ref{eq:interpol}) and the elliptic regularity estimate (\ref{eq:ellreg1}). For the opposite bound we provide an indirect argument which -- via blow-up -- relates the desired boundary inequality to the weak unique continuation principle. We carry out the argument for an eigenfunction; for a solution with an arbitrary $C^1$ potential the estimate is of the same spirit. In particular, the following argument then yields the doubling property on the boundary for sufficiently small radii.\\

\subsection{Proof of Proposition \ref{prop:eigen}}

In order to identify the order of vanishing, we, in particular, have to show that the bulk contributions can be controlled via the boundary contributions for sufficiently small radii. A weak statement of this result would be the existence of a constant $C>0$ (possibly depending on $\lambda$) such that
\begin{align*}
\limsup\limits_{r\rightarrow 0} \frac{\left\| y_{n+1}^{\frac{1-2s}{2}}w \right\|_{L^2(B_{r}^+)}}{\left\| w\right\|_{L^2(B_r)}} \leq C <\infty.
\end{align*}
This is a weak form of a (localized) bound for the (generalized) harmonic/ Caffarelli-Silvestre extension. 
We provide an indirect argument which  illustrates the connection of the desired bound and the weak unique continuation principle. Due to its indirect nature, this first argument for the boundary doubling property is relatively weak and not very quantitative. Yet, as it is a relatively straight forward application of the previous results we included it in our discussion. Relying on ideas of Bellova and Lin \cite{BL14}, we present a much more precise estimate in the next sections.

\begin{proof}[Proof of Proposition \ref{prop:eigen}]
Setting $R\sim r_0(g)$, we take the bulk estimate
\begin{align}
\label{eq:bulkd}
\left\| y_{n+1}^{\frac{1-2s}{2}} w \right\|_{L^2(B_{R}^+)} \leq \left( \frac{R}{r} \right)^{C(\lambda^{\frac{1}{2s}}+1)} \left\| y_{n+1}^{\frac{1-2s}{2}} w \right\|_{L^2(B_{r}^+)}
\end{align}
as our starting point and begin by estimating its left hand side by an analogous boundary contribution from below. By the interpolation estimate from \cite{Rue}, the elliptic regularity estimate (\ref{eq:ellreg1}), the doubling estimate and choosing $\tau \sim \lambda^{\frac{1}{2s}}$, we have
\begin{align*}
\left\| w \right\|_{L^2(B_{R})} & \lesssim \tau^{1-s} \left\| y_{n+1}^{\frac{1-2s}{2}} w \right\|_{L^2(B_{2R}^+)} +
 \tau^{-s} \left\| y_{n+1}^{\frac{1-2s}{2}} \nabla w \right\|_{L^2(B_{2R}^+)} \\
& \lesssim R^{-1} \tau^{1-s} \lambda \left\| y_{n+1}^{\frac{1-2s}{2}} w \right\|_{L^2(B_{4R}^+)}\\
& \lesssim e^{C(\lambda^{\frac{1}{2s}}+1)} \left\| y_{n+1}^{\frac{1-2s}{2}}w \right\|_{L^2(B_{R}^+)}.
\end{align*}
This yields the desired lower bound for the left hand side of (\ref{eq:bulkd}).
Therefore it remains to deal with the right hand side of (\ref{eq:bulkd}). To this end, we prove the following statement: For any fixed $w$ there exists a constant $C>0$ (in particular possibly $\lambda$ dependent) and a radius $\bar{r}$ with $0<\bar{r}\leq R$ such that
\begin{align}
\label{eq:ext}
\left\| y_{n+1}^{\frac{1-2s}{2}}w\right\|_{L^2(B_{r}^+)} \leq C \left\| w \right\|_{L^2(B_{r})} \mbox{ for all } 0<r\leq \bar{r}. 
\end{align} 
Combined with the estimate from above, this then implies the desired result on the vanishing order:
\begin{align*}
\left\| w \right\|_{L^2(B_{R})} \leq C(\lambda, R) \left( \frac{R}{r} \right)^{C(\lambda^{\frac{1}{2s}} + 1)} \left\| w \right\|_{L^2(B_{r})}.
\end{align*}
In order to establish (\ref{eq:ext}), we argue by contradiction: Assuming the statement were wrong, a diagonal argument would yield the existence of a sequence of radii $\{r_k\}_{k\in \N}$, $r_k>0$, $r_k \rightarrow 0$ such that
\begin{align}
\label{eq:assum}
k \left\| w \right\|_{L^2(B_{r_k})} <  \left\| y_{n+1}^{\frac{1-2s}{2}} w \right\|_{L^2(B_{r_k}^+)}.
\end{align}
By blowing up the respective functions $w$, we illustrate that this then contradicts the weak unique continuation principle. Indeed, consider
\begin{align*}
w_{r_k}(y):= \frac{w(r_k y)}{ r_k^{-\frac{n+1}{2}} r_k^{\frac{2s-1}{2}} \left\| y_{n+1}^{\frac{1-2s}{2}} w \right\|_{L^2(B_{r_k}^+)} }.
\end{align*}
For this function we have the following estimates:
\begin{itemize}
\item $\left\| y_{n+1}^{\frac{1-2s}{2}} w_{r_k}  \right\|_{L^2(B_{1}^+)}=1$,
\item $\left\| y_{n+1}^{\frac{1-2s}{2}} \nabla w_{r_k}  \right\|_{L^2(B_{1}^+)}\leq C(\lambda)$,
\item $\left\| w_{r_k} \right\|_{L^2(B_1)} \leq  r_k^{1-s}k^{-1}$.
\end{itemize}
Here the second estimate follows as a consequence of the doubling property and a Caccioppoli/elliptic regularity estimate in the form of (\ref{eq:ellreg1}):
\begin{align*}
\left\| y_{n+1}^{\frac{1-2s}{2}} \nabla w \right\|_{L^2(B_{r_k}^+)} & \lesssim r_k^{-1} \lambda^{3-2s} \left\| y_{n+1}^{\frac{1-2s}{2}} w\right\|_{L^2(B_{2r_k}^+)}\\
& \lesssim r_k^{-1}\lambda^{3-2s} e^{C(\lambda^{\frac{1}{2s}}+1)} \left\| y_{n+1}^{\frac{1-2s}{2}} w\right\|_{L^2(B_{r_k}^+)},
\end{align*}
where we made use of the doubling inequality.
The third property follows from a rescaling argument and our assumption (\ref{eq:assum}):
\begin{align*}
\left\| w_{r_k}\right\|_{L^2(B_1)} &= \frac{\left\| w(r_k \cdot) \right\|_{L^2(B_1)}}{r^{- \frac{n+1}{2}}_k r^{\frac{2s-1}{2}}_k\left\| y_{n+1}^{\frac{1-2s}{2}} w \right\|_{L^2(B_{r_k}^+)} }\\
& \lesssim \frac{r^{- \frac{n}{2}}_k \left\| w \right\|_{L^2(B_{r_k})} }{r^{- \frac{n+1}{2}}_k r^{\frac{2s-1}{2}}_k k \left\| w \right\|_{L^2(B_{r_k})}} = r_k^{1-s}k^{-1}.
\end{align*}
Thus, via Rellich's compactness theorem we may conclude that $w_{r_k} \rightarrow w_0$ strongly in $L^2(y_{n+1}^{\frac{1-2s}{2}}d\vol_g d y_{n+1}, B_1^+)$, weakly in $H^{1}(y_{n+1}^{\frac{1-2s}{2}}d \vol_g d y_{n+1}, B_{1}^+)$ and $ w_0 = 0$ on $B_{1}^+ \cap \{y_{n+1}=0\}$. Moreover, it is a weak solution of
\begin{align*}
(\p_{n+1}y_{n+1}^{1-2s} \p_{n+1} + y_{n+1}^{1-2s}\D_{g})w_0 &= 0 \mbox{ on } B_{1}^+,\\
w_0 &= 0 \mbox{ on } B_{1}^+ \cap \{y_{n+1}=0\},\\
\lim\limits_{y_{n+1}\rightarrow 0}y_{n+1}^{1-2s} \p_{n+1}w_0 &= 0 \mbox{ on } B_{1}^+ \cap \{y_{n+1}=0\}.
\end{align*}
Thus, the weak unique continuation principle asserts $w_0=0$ which contradicts $\left\| y_{n+1}^{\frac{1-2s}{2}} w_0 \right\|_{L^2(B_{1}^+)}=1$. Hence, (\ref{eq:assum}) must be wrong, which proves the claimed estimate (\ref{eq:ext}).
Finally, the claim of Proposition \ref{prop:eigen} follows from the finite multiplicity of the eigenvalues (in the case of $C^1$ potentials the finite multiplicity is replaced by the Fredholm alternative). 
\end{proof}

Last but not least we present the argument for Corollary \ref{cor:vani}.

\begin{proof}
The corollary follows from the estimate in the previous proof in combination with the fact that there each eigenvalue only occurs with a finite multiplicity. 
\end{proof}

\section{An Improved Boundary Estimate}

In this section we follow the ideas of Bellova and Lin \cite{BL14} to transfer the bulk estimates into a doubling condition on the boundary. Although these ideas are not new, in combination with our bulk inequalities they permit to deduce a strong boundary estimate. As we are not forced to work on small $\lambda$-dependent geodesic balls, this can then be exploited to establish (optimal) upper bounds on nodal domains for the generalized Dirichlet-to-Neumann map on compact, real analytic Riemannian manifolds without boundary. We point out that, in particular, this procedure is much more precise than the blow-up argument of the previous section and also allows to give another (improved) proof of the estimates on the vanishing order.\\
The main result of this section reads:

\begin{prop}
\label{prop:bdoubl}
Let $(M,g)$ be a smooth $n$-dimensional Riemannian manifold without boundary, let $s\in (0,1)$ and suppose that $V\in C^{1}(M)$. Assume that $w: M \times \R_+ \rightarrow \R$ is a solution of $\Lambda w = V w$, which is to be understood in the sense of (\ref{eq:CSE}), (\ref{eq:GDN}).
Then there exists a radius $0<r_0<\infty$ which only depends on the manifold $(M,g)$ and the constants $s$, $\left\| b \right\|_{L^{\infty}}$, $\left\| c \right\|_{L^{\infty}}$ such that for all radii $0<r \leq r_0$ and all points $p\in M$ the following boundary doubling inequality is true:
\begin{align*}
\int\limits_{B_{2 r}(p)} w^2 dx \leq e^{C(\left\| V \right\|_{C^1}^{\frac{1}{2s}} + 1)} \int\limits_{B_{r}(p)} w^2 dx.
\end{align*}
The constant $C$ only depends on the manifold $(M,g)$ and the constants $s$, $\left\| b \right\|_{L^{\infty}}$, $\left\| c \right\|_{L^{\infty}}$.
\end{prop}

\begin{rmk}
As above we emphasize that the maximal radius $r_0>0$ on which our doubling property holds true is independent of the eigenvalue $\lambda$ and the potential $V$. As in the case of the Laplacian on a manifold, it is a quantity which is purely determined by the coefficients of the manifold and of the equation. This remains true for the application to the Steklov problem.
\end{rmk}

In order to prove this result, we rely on two auxiliary lemmata which are also used by Bellova and Lin. The first is a slight generalization of the smallness estimate from \cite{BL14}:

\begin{lem}
\label{lem:small}
Let $w: \R^{n+1}_+ \rightarrow \R $ be a solution of 
\begin{align}
\label{eq:small}
(\p_{n+1} y_{n+1}^{1-2s} \p_{n+1} + y_{n+1}^{1-2s}\D_g + y_{n+1}^{1-2s}b \cdot \nabla + y_{n+1}^{1-2s}c)w = 0 \mbox{ on } B_{1}^+
\end{align}
with $\left\| y_{n+1}^{\frac{1-2s}{2}} w \right\|_{L^2(B_{1}^+)}\leq 1$ and $0 \leq \left\| b \right\|_{L^{\infty}(B_{1}^+)}, \left\| c \right\|_{L^{\infty}(B_{1}^+)} \leq K <\infty$.
Assume that
\begin{align*}
\left\|  w \right\|_{H^1(\Gamma)} + \left\| \lim\limits_{y_{n+1}\rightarrow 0} y_{n+1}^{1-2s} \p_{n+1} w \right\|_{L^2(\Gamma)} \leq \epsilon \ll 1,
\end{align*}
where $\Gamma = B_{\frac{3}{4}} \times \{0\}$.
Then there exist constants $C>0$, $\beta>0$ which only depend on 
$s$, $n$ and $K$ such that
\begin{align*}
\left\| y_{n+1}^{\frac{1-2s}{2}} w \right\|_{L^2(B_{\frac{1}{2}}^+)} \leq C \epsilon^{\beta}.
\end{align*} 
More precisely, there exists $\alpha \in (0,1)$ such that
\begin{align}
\label{eq:quant} 
\left\| y_{n+1}^{\frac{1-2s}{2}} w \right\|_{L^2(B_{\frac{1}{2}}^+)} \lesssim \left\| y_{n+1}^{\frac{1-2s}{2}} w \right\|_{L^2(B_{1}^+)}^{\alpha}\left( \left\|  w \right\|_{H^1(\Gamma)} + \left\| \lim\limits_{y_{n+1}\rightarrow 0} y_{n+1}^{1-2s} \p_{n+1} w \right\|_{L^2(\Gamma)} \right)^{1-\alpha}.
\end{align} 
\end{lem}

Secondly, as Bellova and Lin, we make use of an interpolation estimate which relates the boundary and the bulk contributions:

\begin{lem}[Interpolation]
\label{lem:interp}
Let $s\in(0,1)$ and $w\in H^{2}(y_{n+1}^{1-2s}d y, \R^{n+1}_+)\cap H^{1}(\R^n)$. Then for any $\eta>0$ it holds
\begin{align*}
\left\| \nabla' w \right\|_{L^2(\R^n)} \lesssim & \  \eta^{1+s} \left( \left\| y_{n+1}^{\frac{1-2s}{2}} \nabla \nabla' w \right\|_{L^2(\R^{n+1}_+)} + \left\| y_{n+1}^{\frac{1-2s}{2}}  w \right\|_{L^2(\R^{n+1}_+)} \right)\\
& + \eta^{- \frac{s+1}{s} } \left\| w \right\|_{L^2(\R^n)}.
\end{align*}
\end{lem}

Using these two results, we can prove the boundary doubling estimate along the lines of \cite{BL14}.

\begin{proof}[Proof of Proposition \ref{prop:bdoubl}]
For simplicity we assume that $r_0\geq 1$, that there are local coordinates for the manifold on patches of the size $r=1$ and prove the estimate in this setting only. For arbitrary radii $0 < r\lesssim 1$ the argument follows by scaling.  
As in the argument of Bellova and Lin we proceed in three steps:\\
\emph{Step 1: Lower bounds on boundary contributions.} Normalizing $w$ such that $\left\| y_{n+1}^{\frac{1-2s}{2}} w\right\|_{L^2(B_{1}^+)}= 1$, we claim that there exists a constant $C>0$ such that we have a lower bound on the boundary contributions of the following form:
\begin{align}
\label{eq:lowb}
\left\| w \right\|_{H^1(B_{\frac{3}{8}})} + \left\| \lim\limits_{y_{n+1} \rightarrow 0} y_{n+1}^{1-2s} \p_{n+1} w\right\|_{L^2(B_{\frac{3}{8}})} \geq e^{- C( \left\| V \right\|_{C^1}^{\frac{1}{2s}} +1 )}.
\end{align}
Indeed, this follows from the result of Lemma \ref{lem:small}: If the claim were not true, then for any $\tilde{C}>0$ Lemma \ref{lem:small} would imply
\begin{align*}
\left\| y_{n+1}^{\frac{1-2s}{2}} w \right\|_{L^2(B_{\frac{3}{8}}^+)} \leq C e^{- \beta \tilde{C}(\left\| V \right\|_{C^1}^{\frac{1}{2s}} +1)}.
\end{align*}
Recalling the (bulk) normalization of $w$ in $B_{1}^+$ and the bulk doubling inequality
\begin{align*}
e^{-C(\left\| V\right\|_{C^1} + 1)} \lesssim e^{-C(\left\| V\right\|_{C^1} + 1)} \left\| y_{n+1}^{\frac{1-2s}{2}} w \right\|_{L^2(B_{1}^+)}  \lesssim \left\| y_{n+1}^{\frac{1-2s}{2}} w \right\|_{L^2(B_{\frac{3}{8}}^+)},
\end{align*}
this would however yield a contradiction if $\tilde{C}$ is chosen sufficiently large in dependence of $\beta$ and $\left\| V \right\|_{C^1(M)}$.\\
\emph{Step 2: Lower bounds on the $L^2$ boundary contributions.}  We claim that there exists a constant $0<C< \infty$ depending only on $M$ such that
\begin{align*}
\left\| w \right\|_{L^2(B_{\frac{3}{8}})} \geq (1+\left\| V \right\|_{C^1})^{-1}e^{-C(\left\| V \right\|_{C^1}^{\frac{1}{2s}} + 1)}.
\end{align*}
The argument for this claim will on the one hand make use of the fact that $w$ is a solution of $\Lambda w = V w$ and of maximal regularity estimates for the generalized Caffarelli-Silvestre extension operator on the other hand. As $w$ is a solution of $\Lambda w = V w$, (\ref{eq:lowb}) immediately turns into
\begin{align*}
\left\| w \right\|_{\dot{H}^1(B_{\frac{3}{8}})} + \left\| V \right\|_{C^1} \left\|  w\right\|_{L^2(B_{\frac{3}{8}})} \geq e^{- C( \left\| V \right\|_{C^1}^{\frac{1}{2s}} +1 )},
\end{align*}
which leads to
\begin{align}
\label{eq:lowb1}
\left\| w \right\|_{\dot{H}^1(B_{\frac{3}{8}})} +  \left\|  w\right\|_{L^2(B_{\frac{3}{8}})} \geq (1+\left\| V \right\|_{C^1})^{ -1} e^{- C( \left\| V \right\|_{C^1}^{\frac{1}{2s}} +1 )}.
\end{align}
Thus, in order to obtain the desired lower $L^2$ bound, it suffices to control the homogeneous $H^1$ norm of $w$. This is achieved via the interpolation estimate, Lemma \ref{lem:interp}. We define $\tilde{w}:= w\varphi$, where $\varphi$ is a radial cut-off function which is equal to one on $B_{\frac{7}{8}}^+$ and vanishes outside $B_{1}^+$.
Employing the interpolation inequality of Lemma \ref{lem:interp}, then yields
\begin{align*}
\left\| \nabla' \tilde{w} \right\|_{L^2(\R^n)}
\leq & \ c \left(\eta^{1+s} \left(\left\| y_{n+1}^{\frac{1-2s}{2}} \nabla \nabla' \tilde{ w} \right\|_{L^2(\R^{n+1}_+)} + \left\| y_{n+1}^{\frac{1-2s}{2}} \tilde{ w} \right\|_{L^2(\R^{n+1}_+)} \right) \right.\\
&\left. + \eta^{- \frac{s+1}{s}} \left\| \tilde{w} \right\|_{L^2(\R^{n})}\right)\\
\leq & \ c (\eta^{1+s}  \left\| V \right\|_{L^{\infty}} \left\| \nabla' \tilde{w} \right\|_{L^2(\R^n)} + \eta^{1+s} \left\|\nabla' V \right\|_{L^{\infty}} \left\| \tilde{w} \right\|_{L^2(\R^n)} \\
&\quad  \quad  + \eta^{1+s}\left\| y_{n+1}^{\frac{1-2s}{2}} \tilde{w} \right\|_{L^2(\R^{n+1}_+)} 
+ \eta^{- \frac{s+1}{s}} \left\| \tilde{w} \right\|_{L^2(\R^{n})}).
\end{align*}
Here we made use of Lemma \ref{lem:boundary1} in order to bound the second order derivatives.
For $0<\eta = \eta(\left\| V\right\|_{L^{\infty}}) \ll 1$, which is to be fixed later, this implies
\begin{align*}
\left\| \nabla' \tilde{w} \right\|_{L^2(\R^n)}
&\leq c (\eta^{1+s}  \left\| \nabla' V \right\|_{L^{\infty}} \left\| \tilde{w} \right\|_{L^2(\R^n)}\\
& \quad \quad + \eta^{1+s}\left\| y_{n+1}^{\frac{1-2s}{2}} \tilde{w} \right\|_{L^2(\R^{n+1}_+)} + \eta^{- \frac{s+1}{s}} \left\| \tilde{w} \right\|_{L^2(\R^{n})}),
\end{align*}
Hence,
\begin{align*}
\left\| \nabla' w \right\|_{L^2(B_{\frac{3}{8}})} &\leq c \left\| \nabla' \tilde{w} \right\|_{L^2(\R^n)}\\
&\leq c (\eta^{1+s} \left\| \nabla' V\right\|_{L^{\infty}(\R^n)} \left\| \tilde{w} \right\|_{L^2(\R^n)} + \eta^{1+s}\left\| y_{n+1}^{\frac{1-2s}{2}} \tilde{w} \right\|_{L^2(\R^{n+1}_+)} \\ 
& \quad \quad+ \eta^{- \frac{s+1}{s}} \left\| \tilde{w} \right\|_{L^2(\R^{n})})\\
&\leq c( \eta^{1+s}  \left\| y_{n+1}^{\frac{1-2s}{2}} w \right\|_{L^2(B_{1}^+)}  + \eta^{- \frac{s+1}{s}}( \left\| \nabla' V \right\|_{L^{\infty}} +1) \left\| w \right\|_{L^2(B_{\frac{7}{16}})} )\\
&\leq c( \eta^{1+s}   + \eta^{- \frac{s+1}{s}}( \left\| \nabla' V \right\|_{L^{\infty}} +1) \left\| w \right\|_{L^2(B_{\frac{7}{16}})} ).
\end{align*} 
Thus, choosing $\eta$ so small that $c \eta^{1+s} = \frac{1}{2} (1+\left\| V \right\|_{C^1})^{-1}e^{-C(\left\| V \right\|_{C^1}^{\frac{1}{2s}}+1)}$ and using (\ref{eq:lowb1}), then yields
\begin{align*}
(1+\left\| V \right\|_{C^1})^{-1} e^{-C(\left\| V \right\|_{C^1}^{\frac{1}{2s}} +1)} \leq & \  e^{C_s (\left\| V \right\|_{C^1}^{\frac{1}{2s}}+1) }\left\| w \right\|_{L^2(B_{\frac{7}{16}})} \\
&+  \frac{1}{2}(1+ \left\| V \right\|_{C^1})^{-1} e^{-C(\left\| V \right\|_{C^1}^{\frac{1}{2s}} +1)} .
\end{align*}
Absorbing the factors which are only polynomial in $\left\| V \right\|_{C^1}$ into the contributions which are exponential in $\left\| V \right\|_{C^1}$ proves the claim.\\
\emph{Step 3: Boundary Doubling.}
Let $\tilde{\varphi}$ be a cut-off function which is one for $0<|y|\leq \frac{7}{8}$ and which is zero for $|y|\geq \frac{15}{16}$. Due to the Hardy-trace inequality and the elliptic estimates for $w$ (c.f. equation (\ref{eq:ellreg1})), we infer
\begin{align*}
\left\| w \right\|_{L^2(B_{\frac{7}{8}})} &\lesssim \left\| \tilde{\varphi}w \right\|_{L^2(B_{\frac{7}{8}})} 
\lesssim \left\| y_{n+1}^{1-2s} \nabla (\tilde{\varphi}w) \right\|_{L^2(\R^{n+1}_+)} \\
& \lesssim  \left\| y_{n+1}^{\frac{1-2s}{2}} \nabla w\right\|_{L^2(B_{\frac{15}{16}}^+)} + \left\| y_{n+1}^{\frac{1-2s}{2}}  w\right\|_{L^2(B_{\frac{15}{16}}^+)} \\
& \lesssim  C(\left\| V \right\|_{C^1},s) \left\| y_{n+1}^{1-2s}  w \right\|_{L^2(B_1^+)} \lesssim C(\left\| V \right\|_{C^1},s). 
\end{align*}
Again absorbing the polynomial contributions in $\left\| V \right\|_{C^1}$ into the exponential factor, results in the doubling inequality:
\begin{align*}
\left\| w \right\|_{L^2(B_{\frac{7}{8}})} \lesssim e^{C (\left\| V \right\|_{C^1}^{\frac{1}{2s}}+1)} \left\| w \right\|_{L^2(B_{\frac{7}{16}})}.
\end{align*}
\end{proof}

Thus, it remains to provide the proofs of the Lemma \ref{lem:small} and Lemma \ref{lem:interp}. We begin with the smallness lemma. It can be deduced as a consequence of a boundary Carleman inequality (c.f. \cite{LZ98}, \cite{LRob}) similar to our main estimate. Instead of using a weight which explodes at the origin, we however now use a weight which attains a finite value at the origin (in our case $\phi(0)=0$) and decreases in the radial variable.

\begin{proof}[Proof of Lemma \ref{lem:small}]
It suffices to prove estimate (\ref{eq:quant}). This can be deduced via a Carleman estimate. \\
\emph{Step 1: Carleman estimate.} Consider $w: \R^{n+1}_+ \rightarrow \R$ with $w\in C^{\infty}_{0}(\overline{B_{1}^+})$. Let $\phi: \R^{n+1}_+ \rightarrow \R$ be a pseudoconvex weight function with respect to the operator
\begin{align*}
L_s := \p_{n+1}y_{n+1}^{1-2s}\p_{n+1} + y_{n+1}^{1-2s}\D_{g}.
\end{align*}
Then we claim that the following Carleman estimate holds:
\begin{align*}
&\tau^{\frac{3}{2}} \left\| e^{\tau \phi} (y\cdot \nabla \phi)(y \cdot \nabla (y\cdot \nabla \phi)) |y|^{-1} y_{n+1}^{\frac{1-2s}{2}} w \right\|_{L^2(\R^{n+1}_+)} \\
&+ 
\tau^{\frac{1}{2}} \left\| e^{\tau \phi} (y\cdot \nabla \phi)(y \cdot \nabla (y\cdot \nabla \phi)) y_{n+1}^{\frac{1-2s}{2}} \nabla w \right\|_{L^2(\R^{n+1}_+)}\\
 \lesssim & \ \left\| e^{\tau \phi} |y| y_{n+1}^{\frac{2s-1}{2}} (\p_{n+1}  y_{n+1}^{1-2s} \p_{n+1} + y_{n+1}^{1-2s} \D_g) w \right\|_{L^2(\R^{n+1}_+)}\\
&+ \tau^{\frac{1}{2}} \left\| e^{\tau \phi} |y \cdot \nabla' \phi |^{\frac{1}{2}} w   \right\|_{L^2(\R^n)}
 +  \tau^{\frac{1}{2}} \left\| e^{\tau \phi} |y \cdot \nabla' \phi |^{\frac{1}{2}} |y| \nabla' w   \right\|_{L^2(\R^n)}\\
&+ \tau^{\frac{1}{2}} \left\| e^{\tau \phi} |y \cdot \nabla' \phi |^{\frac{1}{2}} \lim\limits_{y_{n+1}\rightarrow 0} y_{n+1}^{1-2s} \p_{n+1} w   \right\|_{L^2(\R^n)}.
\end{align*}
For this purpose we carry out the same computations as those in the proof of Proposition \ref{prop:Carl}. However, instead of using the equation to estimate the boundary errors, we keep them on the right hand side of the Carleman estimate. In Cartesian coordinates these boundary contributions amount to terms of the form
\begin{align*}
& \tau \int\limits_{\R^{n}} (y\cdot \nabla' \phi) w \lim\limits_{y_{n+1}\rightarrow 0} y_{n+1}^{1-2s} \p_{n+1}w dy, \\
& \tau \int\limits_{\R^{n}} (y\cdot \nabla' \phi) (y \cdot \nabla' w) \lim\limits_{y_{n+1}\rightarrow 0} y_{n+1}^{1-2s} \p_{n+1}w dy, \\
& \tau \int\limits_{\R^{n}} (y\cdot \nabla' (y\cdot \nabla' \phi)) w \lim\limits_{y_{n+1}\rightarrow 0} y_{n+1}^{1-2s}\p_{n+1} w dy.
\end{align*}
Thus, they can be estimated by the following contributions
\begin{align*}
&\tau \left\| |y \cdot \nabla' \phi|^{\frac{1}{2}} w \right\|_{L^2(\R^n)}^2, \  \tau \left\| |y \cdot \nabla' \phi|^{\frac{1}{2}} y\cdot \nabla' w \right\|_{L^2(\R^n)}^2, \\ 
&\tau \left\| |y \cdot \nabla' \phi|^{\frac{1}{2}} \lim\limits_{y_{n+1} \rightarrow 0} y_{n+1}^{1-2s}\p_{n+1}  w \right\|_{L^2(\R^n)}^2.
\end{align*}
This yields the desired estimate. We remark that it remains true for operators with lower order contributions as long as they can be absorbed in the left hand side.\\
\emph{Step 2: Proving the interpolation estimate (\ref{eq:interpol}).}
In order to prove the interpolation estimate, we consider a weight function which decays in the radial variable and vanishes at $y=0$, e.g. $\phi(y) = e^{- \beta |y|} - 1$ for a sufficiently large fixed $\beta$. Defining $\tilde{w}:= w \varphi $, where $\varphi$ is a radial cut-off function which equals one for $|y|\leq \frac{2}{3}$ and vanishes for $|y|\geq \frac{3}{4}$ and using Caccioppoli's estimate leads to the bound
\begin{align*}
e^{\tau \phi(\frac{1}{2})} \left\| w \right\|_{L^2(B_{\frac{1}{2}}^+)} 
\lesssim & \ e^{\tau \phi(\frac{2}{3}) }\left\| w \right\|_{L^2(B_{1}^+)} \\
&+  \left\| w   \right\|_{L^2(B_{\frac{3}{4}})}+   \left\| \nabla' w   \right\|_{L^2(B_{\frac{3}{4}})}
+ \left\|  \lim\limits_{y_{n+1}\rightarrow 0} y_{n+1}^{1-2s} \p_{n+1} w   \right\|_{L^2(B_{\frac{3}{4}})}.
\end{align*}
Multiplying with $e^{-\tau \phi(\frac{1}{2})}$ and optimizing the right hand side in $\tau$ yields (\ref{eq:quant}).
\end{proof}

\begin{proof}[Proof of Lemma \ref{lem:interp}]
We prove the interpolation estimate by first applying the Hardy-trace estimate for functions with support in the unit ball, then rescaling to arbitrary domains and finally ``optimizing'' in the resulting parameter.\\
\emph{Step 1: Hardy-trace inequality.} We claim that for all $w\in H^{2}(y_{n+1}^{1-2s}dy, B_{\mu}^+)$ the following interpolation inequality holds:
\begin{equation}
\begin{split}
\label{eq:addint}
\left\| \nabla' w \right\|_{L^2(B_{\mu})} \lesssim & \ \mu^s \left( \left\| y_{n+1}^{\frac{1-2s}{2}}\nabla \nabla' w \right\|_{L^2(B_{2\mu}^+)} +\left\| y_{n+1}^{\frac{1-2s}{2}}w \right\|_{L^2(B_{2\mu}^+)}\right)\\
& + \mu^{-1}\left\| w \right\|_{L^2(B_{2\mu})}.
\end{split}
\end{equation}
Indeed, first consider
$w \in  H^{2}(y_{n+1}^{1-2s}dy, B_{1}^+)\cap C^{\infty}_{0}(B_{1}^+)$. Then the Hardy-trace inequality implies
\begin{align*}
\left\| \nabla' w \right\|_{L^2(B_{1})} &\leq  \left\| |y|^{-s} \nabla' w  \right\|_{L^2(B_{1})} \lesssim \left\| y_{n+1}^{\frac{1-2s}{2}} \nabla \nabla' w \right\|_{L^2(B_{1}^+)}\\
& \lesssim \left\| y_{n+1}^{\frac{1-2s}{2}} \nabla \nabla' w \right\|_{L^2(B_{1}^+)}  + \left\| w\right\|_{L^2(B_{1})}. 
\end{align*}
Thus, by rescaling, (\ref{eq:addint}) holds for any $w\in H^{2}(y_{n+1}^{1-2s}dy, \R^{n+1}_+)\cap C^{\infty}_{0}(B_{\mu}^+)$. In order to dispose of the compact support condition, we make use of a cut-off argument which then yields
\begin{align*}
\left\| \nabla' w \right\|_{L^2(B_{\mu})} \lesssim & \ \mu^s \left( \left\| y_{n+1}^{\frac{1-2s}{2}}\nabla \nabla' w \right\|_{L^2(B_{2\mu}^+)} + \left\| y_{n+1}^{\frac{1-2s}{2}}w \right\|_{L^2(B_{2\mu}^+)} \right) \\
&+ \mu^{-1}\left\| w \right\|_{L^2(B_{2\mu})}.
\end{align*}
Thus, the full statement of (\ref{eq:addint}) follows.\\
\emph{Step 2: The one-dimensional interpolation inequality.} We claim that for all $w \in H^2(y_{2}^{1-2s}dy, \R^{2}_+)\cap C^{\infty}_0(\overline{\R^{2}_+})$ we have the following multiplicative interpolation inequality
\begin{align}
\label{eq:int1D}
\left\| \nabla w' \right\|_{L^2(\R)} \lesssim \left(  \left\| y_{2}^{\frac{2s-1}{2}} \nabla \nabla' w \right\|_{L^2(\R^2_+)} +  \left\| y_{2}^{\frac{2s-1}{2}} w \right\|_{L^2(\R^2_+)}\right)^{\frac{1}{1+s}}  \left\| w \right\|_{L^2(\R)}^{\frac{s}{s+1}}.
\end{align}
We prove this one-dimensional interpolation inequality by distinguishing two cases and by remarking that (\ref{eq:addint}) remains valid for intervals: More precisely, for any $a\in \R$ it holds:
\begin{equation}
\label{eq:addint1}
\tag{\ref{eq:addint}'}
\begin{split}
\left\| \nabla' w \right\|_{L^2([a,a+\mu])} \lesssim & \ \mu^s \left( \left\| y_{2}^{\frac{1-2s}{2}}\nabla \nabla' w \right\|_{L^2([a-\mu,a+2\mu]\times [0,2\mu])} \right. \\
& \left. + \left\| y_{2}^{\frac{1-2s}{2}} w \right\|_{L^2([a-\mu,a+2\mu]\times [0,2\mu])} \right)+ \mu^{-1}\left\| w \right\|_{L^2([a-\mu,a+2\mu])}.
\end{split}
\end{equation}
Let $I=[a,b]$ be an arbitrary but fixed interval and let $w\in  H^2(y_{2}^{1-2s}dy, \R^{2}_+)\cap C^{\infty}_0(\overline{\R^{2}_+})$. We now distinguish two cases:
\begin{itemize}
\item[Case 1:] Setting $I_1:= \left( a, a+ \frac{|I|}{N} \right)$, $2I_{1}:=\left( a - 2\frac{|I|}{N}, a+ 2\frac{|I|}{N} \right)$ and $I_1^+:= 2I_{1}\times [0,2|I_1|]$, for some $N\in \N$ we assume first that
\begin{align}
\label{eq:case1}
|2I_1|^s \left( \left\| y_{2}^{\frac{1-2s}{2}}\nabla \nabla' w \right\|_{L^2(2I_1^+)} + \left\| y_{2}^{\frac{1-2s}{2}} w \right\|_{L^2(2I_1^+)} \right) \geq |2I_1|^{-1}\left\| w \right\|_{L^2(2I_1)}.
\end{align}
Therefore, in this case we obtain the estimate
\begin{align*}
\left\| \nabla' w \right\|_{L^2(I_1)} \lesssim \left( \frac{|I|}{N} \right)^{s} \left( \left\| y_{2}^{\frac{1-2s}{2}}\nabla \nabla' w \right\|_{L^2(2I_1^+)} + \left\| y_{2}^{\frac{1-2s}{2}}w \right\|_{L^2(2I_1^+)} \right).
\end{align*}
This case will become negligible in the limit $N \rightarrow \infty$.
\item[Case 2:] Now assume that (\ref{eq:case1}) is false on $I_1:=\left(a,a + \frac{|I|}{N}\right)$. Due to the support assumption on $w$ and the monotonicity of the summands of (\ref{eq:addint1}) in $\mu$, there exists $c\in \R$ with $a+ \frac{|I|}{N}<c<\infty$ such that for $U_1:= (a,c)$, $2J_1:= (a - 2\frac{|I|}{N}\ ,c)$, $2J_1^+:=(a - 2\frac{|I|}{N}\,c)\times (0,(c-a))$ the following equality is satisfied:
\begin{align*}
|2J_1|^s \left( \left\| y_{2}^{\frac{1-2s}{2}}\nabla \nabla' w \right\|_{L^2(2J_1^+)} + \left\| y_{2}^{\frac{1-2s}{2}}w \right\|_{L^2(2J_1^+)}  \right)= |2J_1|^{-1}\left\| w \right\|_{L^2(2J_1)}.
\end{align*} 
Hence, in this case the interpolation inequality (\ref{eq:addint1}) can be written as
\begin{align*}
\left\| \nabla' w \right\|_{L^2(J_1)} \lesssim \left( \left\| y_{2}^{\frac{1-2s}{2}} \nabla \nabla' w \right\|_{L^2(2J^+_1)}+ \left\| y_{2}^{\frac{1-2s}{2}}w \right\|_{L^2(2J_1^+)}  \right)^{\frac{1}{1+s}} \left\| w \right\|_{L^2(2J_1)}^{\frac{s}{s+1}}.
\end{align*}
\end{itemize}
We now iterate the described estimation procedure. After at most $N$ steps we are then left with (almost) disjoint intervals $I_i, J_i$ (which are possibly empty) and
\begin{align*}
\left\| \nabla' w \right\|_{L^2(I)} \lesssim & \  \sum\limits_{i=1}^{N} \left( \frac{2|I|}{N} \right)^s \left( \left\| y_{2}^{\frac{1-2s}{2}} \nabla \nabla' w \right\|_{L^2(2I_{i}^+)} + \left\| y_{2}^{\frac{1-2s}{2}}w \right\|_{L^2(2I_i^+)}  \right)\\
&+
\sum\limits_{i=1}^{N} \left( \left\| y_{2}^{\frac{1-2s}{2}}\nabla \nabla' w \right\|_{L^2(2J_i)} + \left\| y_{2}^{\frac{1-2s}{2}}w \right\|_{L^2(2J_i^+)}  \right)^{\frac{1}{1+s}} \left\| w \right\|_{L^2(2J_i)}^{\frac{s}{1+s}}\\
 \lesssim & \ \left( \frac{|I|}{N} \right)^s \left( \left\| y_{2}^{\frac{1-2s}{2}} \nabla \nabla' w \right\|_{L^2(\R^2_+)} + \left\| y_{2}^{\frac{1-2s}{2}}  w \right\|_{L^2(\R^2_+)} \right) \\
&+ \left( \left\| y_{2}^{\frac{1-2s}{2}} \nabla \nabla' w\right\|_{L^2(\R^{2}_+)} + \left\| y_{2}^{\frac{1-2s}{2}}  w \right\|_{L^2(\R^2_+)}\right)^{\frac{1}{1+s}} \left\| w \right\|_{L^2(\R)}^{\frac{s}{s+1}},
\end{align*}
where we used the disjointness of the intervals $I_i$ and $J_i$ for $i\in \{1,...,N\}$ in combination with Hölder's inequality. 
Passing to the limit in $N\rightarrow \infty$ finally yields the desired result.\\
\emph{Step 3: Extension to arbitrary dimensions.} The $n$-dimensional result follows from the one-dimensional result by applying Fubini's theorem and the one-dimensional estimate in one variable while keeping the other variables fixed. Combining this with Hölder's inequality implies the claim. Using the density of $C_{0}^{\infty}$, we finally pass from functions in $w\in  H^2(y_{n+1}^{1-2s}dy, \R^{n+1}_+)\cap C^{\infty}_0(\overline{\R^{n+1}_+}))$ to arbitrary $H^2(y_{n+1}^{1-2s}dy, \R^{n+1}_+)\cap L^{2}(\R^n)$ functions.
\end{proof}

\section[Nodal Domains]{Consequences of the Boundary Doubling Property: Estimates on Nodal Domains.}

In this section we give a sketch of the consequences of the boundary doubling property by providing upper bounds on the $(n-1)$-dimensional Hausdorff measure of the nodal domains of eigenfunctions of the generalized Dirichlet-to-Neumann map on a real analytic manifold $(M,g)$ without boundary. As in \cite{DoF} and \cite{BL14}, this is a consequence of a complexification of the problem. Since this complexification can be achieved on complex balls of the size $\sim 1$, it is then possible to conclude via an integral geometry estimate \cite{HL10}. As these results have been well established since the work of \cite{DoF}, we do not go into detail in most of the arguments.

\begin{prop}[Analytic extension] 
\label{prop:analyt}
Let $(M,g)$ be a real analytic Riemannian manifold without boundary. Let $s\in(0,1)$ and let $w$ be a (generalized) eigenfunction of 
(\ref{eq:CSE}), (\ref{eq:GDN}) with analytic coefficients $b,c$.
Then $w$ has an analytic extension into the complex plane and for $0<r_2 < r_1 < r_0(g) < \infty $ it holds
\begin{align}
\label{eq:ana2}
\sup\limits_{z\in B_{r_2}\subset \C^n}|w(z)| \leq e^{C(\lambda^{\frac{1}{2s}} +1 )} \sup\limits_{x \in B_{r_1}\subset \R^n}|w(x)|.
\end{align}
\end{prop}

\begin{proof}
In the case $s= \frac{1}{2}$ the analyticity can be obtained as a consequence of the theorem of elliptic iterates \cite{LM}. In the case of general $s\in (0,1)$ we employ Corollary \ref{cor:analytic}. In both cases the radius of convergence of the associated power series depends on $\lambda$:
Rescaling implies that $\tilde{w}(y):= w(\lambda^{- \frac{1}{2s}} y)$ satisfies
\begin{align*}
\p_{n+1} y_{n+1}^{1-2s} \p_{n+1} \tilde{w} + y_{n+1}^{1-2s}\D_g \tilde{w} + y_{n+1}^{1-2s} \lambda^{-\frac{1}{2s}}  b \cdot \nabla \tilde{w} +   y_{n+1}^{1-2s} \lambda^{-\frac{1}{s}} c \tilde{w} & = 0 \mbox{ on } B_{1}^+,\\
\lim\limits_{y_{n+1}\rightarrow 0}y_{n+1}^{1-2s} \p_{n+1} \tilde{w} &= \tilde{w} \mbox{ on } B_1. 
\end{align*}
Therefore, for a multi-index $\alpha$ of length $|\alpha|=k$, we have
\begin{align}
\label{eq:analyt}
|\nabla'^\alpha\tilde{w} (p)|\leq C C_s^k k!.
\end{align}
This allows to deduce the statement of the proposition: Indeed, scaling back yields that $w$ can be analytically extended into a complex neighbourhood of the size $B_{r \lambda^{-\frac{1}{2s}}}$ (where $r$ is the radius of convergence of the power series for $\lambda=1$) with an estimate of the form
\begin{align*}
\sup\limits_{z\in B_{\lambda^{-\frac{1}{2s}}r}\subset \C^n}|w(z)| \leq C \sup\limits_{x \in B_{C\lambda^{-\frac{1}{2s}}r}\subset \R^n}|w(x)|,
\end{align*}
for some $C>1$.
Applying this estimate to a translated version of $w$ and iterating this procedure $\lambda^{\frac{1}{2s}}$ times in the complex directions leads to the desired claim.
\end{proof}

Via the doubling property we then obtain the following estimate for the complex extension.

\begin{prop}[Doubling] 
\label{prop:doubl2}
Let $(M,g)$ be a real analytic Riemannian manifold without boundary. Let $s\in (0,1)$ and let $w$ be a (generalized) eigenfunction of 
(\ref{eq:CSE}), (\ref{eq:GDN}) with analytic coefficients $b,c$.
Then $w$ has an analytic extension into the complex plane and for $0<2 r < r_0(g) < \infty $ it holds
\begin{align*}
\sup\limits_{z\in B_{2 r}\subset \C^n}|w(z)| \leq e^{C(\lambda^{\frac{1}{2s}} +1 )} \sup\limits_{x \in B_{ r}\subset \R^n}|w(x)|.
\end{align*}
The constant $C$ only depends on the manifold $(M,g)$ and the constants $s$, $\left\| b \right\|_{L^{\infty}}$, $\left\| c \right\|_{L^{\infty}}$.
\end{prop}

\begin{proof}
We have to transfer the $L^2$ boundary doubling properties of the previous section to $L^{\infty}$ doubling properties. This is achieved by using the elliptic regularity estimates from Section \ref{sec:analytic} in combination with Lemma \ref{lem:small} and the $L^2$ (bulk) doubling properties. In order to obtain the claim of the proposition it suffices to estimate the right hand side (\ref{eq:ana2}) by $ e^{C(\lambda^{\frac{1}{2s}} +1 )} \sup\limits_{x \in B_{ \frac{r_2}{2}}\subset \R^n}|w(x)|$. Moreover, by scaling, we may assume $r_2= \frac{1}{2}$ and $r_1 = C_1r_2$ with $C_1>1$.\\
Using Lemma \ref{lem:small} and in particular estimate (\ref{eq:quant}) in combination with the bulk doubling properties of $w$, we infer
\begin{align*}
\left\| y_{n+1}^{\frac{1-2s}{2}} w \right\|_{L^2(B_{\frac{1}{2}}^+)} \lesssim e^{C( \lambda^{\frac{1}{2s}}+1)} \left( \left\| w \right\|_{H^1(\Gamma)} + \left\| \lim\limits_{y_{n+1}\rightarrow 0} y_{n+1}^{1-2s} \p_{n+1}w \right\|_{L^2(\Gamma)} \right).
\end{align*}
Here we used the notation of Lemma \ref{lem:small}. We point out that the constant $C$ depends on $\alpha$ (which however is a fixed constant). Relying on similar arguments as in step 2 of the proof of Proposition \ref{prop:bdoubl} and once more making use of the bulk doubling property, this estimate can be further improved to read
\begin{align}
\label{eq:boundbulk}
\left\| y_{n+1}^{\frac{1-2s}{2}} w \right\|_{L^2(B_{\frac{1}{2}}^+)} \lesssim e^{C( \lambda^{\frac{1}{2s}}+1)}  \left\| w \right\|_{L^2(\Gamma)} .
\end{align}
As a result, we can then conclude:
\begin{align*}
\left\| w \right\|_{L^{\infty}(B_{1/2}(z))} 
& \leq e^{C(\lambda^{\frac{1}{2s}}+1)}\left\| w \right\|_{L^{\infty}(B_{C_1/2}(x))} 
 \leq e^{C(\lambda^{\frac{1}{2s}}+1)}\left\| y_{n+1}^{\frac{1-2s}{2}} w \right\|_{L^{2}(B_{2 C_1}^+(x))} \\
& \leq e^{C(\lambda^{\frac{1}{2s}}+1)}\left\| y_{n+1}^{\frac{1-2s}{2}} w \right\|_{L^{2}(B_{1/8}^+(x))} 
 \leq e^{C(\lambda^{\frac{1}{2s}}+1)}\left\|  w \right\|_{L^{2}(B_{1/4}(x))} \\
& \leq e^{C(\lambda^{\frac{1}{2s}}+1)}\left\|  w \right\|_{L^{\infty}(B_{1/4}(x))}.
\end{align*}
Here we first used Proposition \ref{prop:analyt}, followed by a variant of Corollary \ref{cor:variablefull} (which in combination with Morrey's inequality allows to control the boundary $L^{\infty}$ norm by the bulk $L^2$ norm), the bulk doubling property and last but not least estimate (\ref{eq:boundbulk}).
\end{proof}

In order to achieve the estimate on the Hausdorff measure of the nodal domains we use the following lemma relating the growth of a complex analytic function with the number of its zeros. This can for instance be found in \cite{BL14}:

\begin{lem}
\label{lem:compl}
Let $w: B_{1}\subset \C \rightarrow \C$ be an analytic function with
\begin{align*}
|w(0)|= 1 \mbox{ and } \left\| w \right\|_{L^{\infty}(B_1)} \leq 2^n,
\end{align*}
for $n\geq 0$.
Then
\begin{align*}
\#\{ z\in B_{\frac{1}{2}} \big| w(z)=0 \} \leq n.
\end{align*}
\end{lem}

With this, we can finally prove the claim of Proposition \ref{prop:nodaldom}. We use the same strategy as Bellova and Lin \cite{BL14}.

\begin{proof}[Proof of Proposition \ref{prop:nodaldom}]
In order to prove the Proposition we combine Lemma \ref{lem:compl} with an integral geometry estimate as in \cite{BL14}. However, due to our strong Carleman estimates we can argue on scales of order one. Let $x_0\in B_{\frac{1}{2}}\subset \R^n$ be the point at which that supremum of $|w|$ in $B_{\frac{1}{2}}\subset \R^n$ is attained. Setting $\tilde{w}_{\theta}(z):= w(x_0 + z\theta)$, where $\theta \in S^{n-1}$, $z\in B_{1}(0)\subset \C$, Proposition \ref{prop:doubl2} and Lemma \ref{lem:compl} imply
\begin{align*}
&\#\{x\in B_{\frac{1}{4}}(0) \subset \R^{n} \big| x- x_0 \mbox{ is parallel to } \theta \mbox{ and } w(x)=0  \} \\
& \leq  \#\{z\in B_{\frac{1}{2}}(0) \subset \C \big| \tilde{w}_{\theta}(z)=0  \} \leq C(\lambda^{\frac{1}{2s}} +1).
\end{align*}
Thus, via an integral geometry estimate \cite{HL10}, we infer
\begin{align*}
\mathcal{H}^{n-1}(\{x \in B_{\frac{1}{2}}(x_0) \big| w(x)=0\}) \lesssim \int\limits_{S^{n-1}} C(\lambda^{\frac{1}{2s}} +1) d\theta \lesssim C(\lambda^{\frac{1}{2s}} +1).
\end{align*}

\end{proof}

\section{Appendix}
In this final section we provide (sketches of) proofs of the required regularity statements. We begin by proving inhomogeneous bulk and boundary estimates and then using these to conclude the necessary analyticity estimates which are needed for the results on the $\mathcal{H}^{n-1}$ measure of the nodal domains of generalized eigenfunctions of the fractional Laplacian.

\subsection{Interior Regularity Estimates}
 
\begin{lem}[Global inhomogeneous estimate]
\label{lem:reg1}
Let $s\in (0,1)$,  $f\in \mathcal{S}\cap L^{2}(y_{n+1}^{\frac{2s-1}{2}}, \R^{n+1}_+)$ and $u$ be a solution of 
\begin{align*}
\nabla \cdot y_{n+1}^{1-2s}\nabla u & = f \mbox{ in } \R^{n+1}_{+},\\
\lim\limits_{y_{n+1}\rightarrow 0} y_{n+1}^{1-2s} \p_{n+1} u & = 0 \mbox{ on } \R^n \times \{0\}.
\end{align*}
Then we have
\begin{align*}
\left\| y_{n+1}^{\frac{1-2s}{2}} D^{2} u \right\|_{L^2(\R^{n+1}_{+})}^2 \lesssim \left\| y_{n+1}^{\frac{2s-1}{2}} f \right\|_{L^2(\R^{n+1}_{+})}^2.
\end{align*}
\end{lem}

\begin{proof}
\emph{Step 1: Tangential transformation and conversion into an ODE.}
We carry out a tangential Fourier transform and interpret the PDE as an ODE in $z= |\xi|y_{n+1}$. Setting $\tilde{w}(z,\xi)=|\xi|^{2s+1} u(z,\xi)$ it evolves according to
\begin{align*}
z^{1-2s} \tilde{w}'' + (1-2s)z^{-2s}\tilde{w}' - z^{1-2s}\tilde{w} & = \hat{f} \mbox{ on } (0,\infty),\\
\lim\limits_{z\rightarrow 0}z^{1-2s} \tilde{w}' & = 0.
\end{align*}

\emph{Step 2: Energy estimates.} Formally testing the resulting ODE with $\tilde{w}$ yields
\begin{align*}
\int\limits_{0}^{\infty} z^{1-2s} \tilde{w}'^2 dz + \int\limits_{0}^{\infty} z^{1-2s} \tilde{w}^2 dz
= - \int\limits_{0}^{\infty} f \tilde{w} dz.
\end{align*}
Applying Cauchy-Schwarz and Young's inequality thus implies
\begin{align*}
\left\| z^{\frac{1-2s}{2}} \tilde{w}' \right\|_{L^2(0,\infty)} + \left\| z^{\frac{1-2s}{2}} \tilde{w}  \right\|_{L^2(0,\infty)}
\lesssim \left\| z^{\frac{2s-1}{2}} f \right\|_{L^2(0,\infty)}.
\end{align*}

\emph{Step 3: Reduction of the ODE.}
As a consequence of the energy estimates, we may interpret our original second order ODE as a first order ODE for $\tilde{w}'$:
\begin{align*}
z^{\frac{1-2s}{2}}\tilde{w}'' + (1-2s)z^{- \frac{1+2s}{2}}\tilde{w}' = z^{\frac{2s-1}{2}} \hat{f} + z^{\frac{1-2s}{2}}\tilde{w} =: \tilde{f}.
\end{align*}
Setting $v=\tilde{w}'$, the resulting first order ode can be solved via its fundamental solution:
\begin{align*}
\tilde{k}(z,y) = \left\{ 
\begin{array}{ll}
y^{\frac{1-2s}{2}}z^{2s-1} & z>y>0,\\
0 & \mbox{else}.
\end{array} \right.
\end{align*}
From this it is possible to obtain weighted $L^{\infty}-L^{\infty}$ and $L^{1}-L^{1}$ estimates:
Using Schur's Lemma, we observe
\begin{align*}
\sup\limits_{z\in (0,\infty)} \int\limits_{(0,\infty)}\left( \frac{z}{y} \right)^{-\delta} z^{- \frac{2s+1}{2}} |\tilde{k}(z,y)|dy = \sup\limits_{z} z^{-1 + \frac{2s-1}{2}-\delta} \int\limits_{0}^{z}y^{ \frac{1-2s}{2}+\delta} dy = const,
\end{align*}
if $s\in (0,1)$ and $\delta > \frac{2s-1}{2} -1$. Therefore, $\left\| z^{-\frac{ 2s+1}{2}-\delta} v \right\|_{L^{\infty}} \lesssim \left\| z^{\delta}\tilde{f} \right\|_{L^{\infty}} $. Similarly,
\begin{align*}
\sup\limits_{y\in(0,\infty)} \int\limits_{(0,\infty)}\left( \frac{z}{y} \right)^{\delta}  z^{- \frac{1+2s}{2}} |\tilde{k}(z,y)|d z = \sup\limits_{y}y^{-\delta} y^{\frac{1-2s}{2}} \int\limits_{y}^{\infty} z^{-1 + \frac{2s-1}{2}+\delta}dz  = const,
\end{align*}
if $s\in(0,1)$ and $\delta < \frac{1-2s}{2}$. This implies $\left\| z^{- \frac{1+2s}{2}+\delta} v\right\|_{L^{1}} \lesssim \left\| z^{-\delta} \tilde{f} \right\|_{L^{1}}$. An admissible choice would, for example, be $\delta= \frac{1-2s}{2}-\epsilon$ with $0<\epsilon=\epsilon(s)\ll 1$. Then interpolation yields the desired weighted $L^{2}$ estimate:
\begin{align}
\label{eq:grad2}
\left\| z^{- \frac{1+2s}{2}}v \right\|_{L^2}^2 \lesssim \left\| \tilde{f}\right\|_{L^2}^2.
\end{align}
Recalling $v= \hat{w}'$, this allows to control the second non-tangential derivatives:
\begin{align*}
z^{\frac{1-2s}{2}}\tilde{w}'' =  z^{\frac{2s-1}{2}}\hat{f} + z^{\frac{1-2s}{2}}\tilde{w} - (1-2s)z^{- \frac{1+2s}{2}}\tilde{w}'.
\end{align*}
Therefore,
\begin{align*}
\left\| z^{\frac{1-2s}{2}} \tilde{w}'' \right\|_{L^2}^2 + \left\| z^{\frac{1-2s}{2}} \tilde{w} \right\|_{L^2}^2 \lesssim \left\| z^{\frac{2s-1}{2}} \hat{f} \right\|_{L^2}^2.
\end{align*}
Finally, rescaling yields the desired estimate
\begin{align*}
\left\| y_{n+1}^{\frac{1-2s}{2}} \p_{n+1}^2 \hat{u} \right\|_{L^2}^2 + \left\| y_{n+1}^{\frac{1-2s}{2}}|\xi|^2 \hat{u} \right\|_{L^2}^2 \lesssim \left\| y_{n+1}^\frac{2s-1}{2} \hat{f} \right\|_{L^2}^2.
\end{align*}
\end{proof}

\begin{lem}[Localization of the inhomogeneous estimate]
\label{lem:reg2}
Let $s\in (0,1)$,  $f\in \mathcal{S}\cap L^{2}(y_{n+1}^{\frac{2s-1}{2}}, B_{1}^+)$ and $u$ be a solution of 
\begin{align*}
\nabla \cdot y_{n+1}^{1-2s}\nabla u & = f \mbox{ in } B_{1}^{+},\\
\lim\limits_{y_{n+1}\rightarrow 0} y_{n+1}^{1-2s} \p_{n+1} u & = 0 \mbox{ on } B_{1} \times \{0\}.
\end{align*}
Then we have
\begin{align*}
\left\| y_{n+1}^{\frac{1-2s}{2}} D^{2} u \right\|_{L^2(B_{\frac{1}{2}}^{+})}^2 \lesssim \left\| y_{n+1}^{\frac{2s-1}{2}} f \right\|_{L^2(B_{\frac{1}{2}}^{+})}^2 + \left\| y_{n+1}^{\frac{1-2s}{2}} u \right\|_{L^2(B_{1}^{+})}^2.
\end{align*}
\end{lem}

\begin{proof}
This follows from Lemma \ref{lem:reg1} after localizing with a radial cut-off function and using Caccioppoli's inequality. 
\end{proof}

\subsection{Boundary Regularity Estimates}

In this section we prove a boundary estimate which allows to make efficient use of the interpolation inequality in the proof of Proposition \ref{prop:bdoubl}

\begin{lem}[Boundary estimate]
\label{lem:boundary}
Let $w: \R^{n+1}_+ \rightarrow \R$ be a solution of
\begin{equation}
\label{eq:pde}
\begin{split}
\p_{n+1}y_{n+1}^{1-2s}\p_{n+1}w + y_{n+1}^{1-2s} \D w &= 0 \mbox{ on } \R^{n+1}_+,\\
\lim\limits_{y_{n+1}\rightarrow 0} y_{n+1}^{1-2s} \p_{n+1} w &= g \mbox{ on } \R^n.
\end{split}
\end{equation}
Then we have:
\begin{align*}
\left\| y_{n+1}^{\frac{1-2s}{2}} \nabla \nabla' w \right\|_{L^2(\R^{n+1}_+)} \lesssim \left\| g \right\|_{\dot{H}^{1-s}(\R^{n})}
\end{align*}
\end{lem}

\begin{proof}

Carrying out a Fourier transform in the tangential direction, transforms the PDE (\ref{eq:pde}) into the following ODE in $y_{n+1}$:
\begin{align*}
y_{n+1}^{1-2s}\hat{w}'' + (1-2s)y_{n+1}^{-2s} \hat{w}' - y_{n+1}^{1-2s}|\xi|^2 \hat{w} &= 0,\\
-\lim\limits_{y_{n+1}\rightarrow 0}y_{n+1}^{1-2s}\hat{w}' & = \hat{g}. 
\end{align*}
Rescaling ($\hat{w}(y_{n+1},\xi) = u(|\xi|y_{n+1},\xi)$, $v = y_{n+1}^{-s}u$, $z=y_{n+1}|\xi|$) leads to a modified Bessel equation
\begin{align*}
z^2v'' + zv' - (z^2 + s^2)v = 0.
\end{align*}
Its two linear independent solutions are given by the modified Bessel functions $I_{s}$ and $K_{s}$ which have the following asymptotics in $0$ and $\infty$ (c.f. \cite{OLBC}):
\begin{align*}
&I_{s}(z) \sim z^{s} \mbox{ as } z \rightarrow 0, \ \ I_{s}(z) \sim z^{-\frac{1}{2}}e^{z} \mbox{ as } z\rightarrow \infty,\\
&K_{s}(z) \sim z^{-s} \mbox{ as } z \rightarrow 0, \ \ K_{s}(z) \sim z^{-\frac{1}{2}}e^{-z} \mbox{ as } z\rightarrow \infty.
\end{align*}
In order to obtain integrability at infinity, we set 
\begin{align*}
\hat{u}(z,\xi)= C(\xi)z^{s}K_{s}(z).
\end{align*}
Taking the boundary conditions and the asymptotics $(z^{s}K_{s}(z))' \sim z^{2s-1}$ as $z\rightarrow 0$ into account, implies
\begin{align*}
\hat{u}(z,\xi) = C|\xi|^{-2s}\hat{g}(\xi)z^{s}K_{s}(z).
\end{align*}
For notational convenience we define $\va(z) = z^{s}K_{s}(z)$. With this we can prove the desired estimate:
\begin{align*}
\left\| y_{n+1}^{\frac{1-2s}{2}}  \nabla \nabla' w \right\|_{L^2(\R^{n+1}_+)}
&\leq  \left\| y_{n+1}^{\frac{1-2s}{2}}  \nabla' \p_{n+1} w \right\|_{L^2(\R^{n+1}_+)} + \left\| y_{n+1}^{\frac{1-2s}{2}}  \nabla'^2 w \right\|_{L^2(\R^{n+1}_+)}\\
& \leq  \left\| z^{\frac{1-2s}{2}} |\xi|^{s+1} \p_z \hat{u} \right\|_{L^2(\R^{n+1}_+)}\\
& \quad \quad + \left\| z^{\frac{1-2s}{2}} |\xi|^{s+1}  \hat{u} \right\|_{L^2(\R^{n+1}_+)}\\
& \leq  \left\|  |\xi|^{-s+1} \hat{g}  \right\|_{L^2(\R^{n})}\left( \left\| z^{\frac{1-2s}{2}} \va' \right\|_{L^2(\R_+)} + \left\| z^{\frac{1-2s}{2}} \va \right\|_{L^2(\R_+)}  \right)\\
& \lesssim \left\| g\right\|_{\dot{H}^{1-s}(\R^n)},
\end{align*}
as the asymptotics of $\va$ allow to integrate the $\R_+$ integral.
\end{proof}

We point out that it is possible to localize this estimate and via perturbation arguments introduce variable coefficients in the tangential direction as well as to add lower order contributions:

\begin{lem}[Localization of the boundary estimate]
\label{lem:boundary1}
Let $a^{ij}: \R^{n} \rightarrow \R^{n\times n}$ be a uniformly elliptic, smooth tensor field depending only on the tangential variables and suppose that $b,c$ are smooth. Let $w: \R^{n+1}_+ \rightarrow \R$ be a solution of
\begin{align*}
\p_{n+1}y_{n+1}^{1-2s}\p_{n+1} + y_{n+1}^{1-2s} \p_i a^{ij}\p_j w + y_{n+1}^{1-2s} b \cdot \nabla w + y_{n+1}^{1-2s}c w &= 0 \mbox{ on } \R^{n+1}_+,\\
\lim\limits_{y_{n+1}\rightarrow 0} y_{n+1}^{1-2s} \p_{n+1} w &= g \mbox{ on } \R^n.
\end{align*}
Then we have:
\begin{align*}
\left\| y_{n+1}^{\frac{1-2s}{2}} \nabla \nabla' w \right\|_{L^2(B^+_{\frac{1}{2}})} \lesssim \epsilon^{\frac{1}{1-s}} \left\| g \right\|_{\dot{H}^{1}(B_1)} +  \epsilon^{- \frac{1}{s}} \left\| g \right\|_{L^2(B_1)} + \left\| y_{n+1}^{1-2s} w \right\|_{L^{2}(B_{1}^+)}
\end{align*}
\end{lem}

Here we made use of an interpolation inequality for the $\dot{H}^{1-s}$ semi-norm.\\

\subsection{Analyticity}
\label{sec:analytic}
Finally, we specialize the previous estimates to the situation of our equation. Assuming smoothness/analyticity of the coefficients then allows to deduce smoothness/analyticity of the solution with respect to the tangential directions.

\begin{prop}[Variable coefficient $L^2$ regularity estimate]
\label{prop:variablefull}
Let $w: M \times \R_+ \rightarrow \R$ be a (generalized) eigenfunction of (\ref{eq:CSE}), (\ref{eq:GDN}) with $\lambda=1$. Then we have
\begin{equation}
\begin{split}
\label{eq:combest}
\left\| \nabla' w \right\|_{L^2(B_{1})} + \left\| y_{n+1}^{\frac{1-2s}{2}} \nabla' \nabla w \right\|_{L^2(B_{1}^+)}
\lesssim \left\| y_{n+1}^{\frac{1-2s}{2}} w \right\|_{L^2(B_{2}^+)} + \left\| w \right\|_{L^2(B_{2})}.
\end{split}
\end{equation}
\end{prop}

\begin{proof}
For notational convenience we assume that $g_{ij}(0)=\delta^{ij}$. Using a radial cut-off function $\eta$ we rewrite our problem as
\begin{align*}
\p_{n+1} y_{n+1}^{1-2s} \p_{n+1} (\eta w) + y_{n+1}^{1-2s} \D' (\eta w) &= y_{n+1}^{1-2s} \p_i (\delta^{ij} - \tilde{g}^{ij}(y)) \p_j (\eta w) + l.o.t.\\
\lim\limits_{y_{n+1} \rightarrow  0 } y_{n+1}^{1-2s} \p_{n+1} (w\eta) &= \eta w,
\end{align*}
for $i\in \{1,...,n\}$.
Here $l.o.t.$ stands for terms of lower order, i.e. terms which involve at most a gradient. Separating the equation into a boundary and a bulk problem we infer from the global bulk and boundary estimates that
\begin{equation}
\label{eq:ana1}
\begin{split}
\left\| y_{n+1}^{\frac{1-2s}{2}} \nabla \nabla' (w\eta) \right\|_{L^2(\R^{n+1}_+)} \lesssim & \ \left\| y_{n+1}^{\frac{1-2s}{2}} \p_i (\delta^{ij} - \tilde{g}^{ij}(y))\p_j (\eta w) \right\|_{L^2(\R^{n+1}_+)} \\
&+ 
\left\| \eta w \right\|_{H^{1-s}(\R^n)} +
l.o.t.
\end{split}
\end{equation}
Using the Hardy-trace as well as Herbst's inequality (\cite{He}), we infer that the left hand side controls the $H^{s+1}$ norm of the boundary data:
\begin{align*}
\left\| \eta w \right\|_{H^{1+s}(\R^n)} \lesssim \left\| y_{n+1}^{\frac{1-2s}{2}} \nabla \nabla' (w\eta) \right\|_{L^2(\R^{n+1}_+)}.
\end{align*}
Thus, interpolation allows to estimate the boundary term on the right hand side by
\begin{align*}
\left\| \eta w \right\|_{H^{1-s}(\R^n)}  \leq \epsilon \left\| \eta w \right\|_{H^{1+s}(\R^n)} + C_{\epsilon} \left\| \eta w \right\|_{L^2(\R^n)}.
\end{align*}
This permits to absorb the higher order contribution into the left hand side of (\ref{eq:ana1}). Choosing the cut-off function with a sufficiently small support, it is moreover possible to also absorb the $ \left\| y_{n+1}^{\frac{1-2s}{2}} \p_i (\delta^{ij} - \tilde{g}^{ij}(y))\p_j (\eta w) \right\|_{L^2(\R^{n+1}_+)} $ contribution in the left hand side due to the continuity of $\tilde{g}^{ij}$.
This yields the desired estimate.
\end{proof}

Using the elliptic estimates (\ref{eq:ellreg1}) we can also formulate Proposition \ref{prop:variablefull} as
\begin{cor}[Variable coefficient $L^2$ regularity estimate]
\label{cor:variablefull}
Let $w: M \times \R_+ \rightarrow \R$ be a solution of (\ref{eq:CSE}), (\ref{eq:GDN}) with $\lambda=1$. Then we have
\begin{equation}
\begin{split}
\label{eq:combest1}
\left\| y_{n+1}^{\frac{1-2s}{2}} \nabla' \nabla w \right\|_{L^2(B_{1}^+)}
\lesssim \left\| y_{n+1}^{\frac{1-2s}{2}} w \right\|_{L^2(B_{3}^+)}.
\end{split}
\end{equation}
\end{cor}

\begin{rmk}
\begin{itemize}
\item We remark that similar estimates hold in the case $\lambda \neq 1$. In this setting the respective constants depend polynomially on $\lambda$.
\item Also in the presence of inhomogeneities similar estimates as in Corollary \ref{cor:variablefull} remain true: If $w: \R^{n+1}_+ \rightarrow \R$ is compactly supported and satisfies (\ref{eq:CSE}), (\ref{eq:GDN}) with an inhomogeneity $f$, we have
\begin{align*}
\left\| y_{n+1}^{\frac{1-2s}{2}} \nabla' \nabla w \right\|_{L^2(\R^{n+1}_+)}
\lesssim \left\| y_{n+1}^{\frac{1-2s}{2}} w \right\|_{L^2(\R^{n+1}_+)} +\left\| y_{n+1}^{\frac{2s-1}{2}} f \right\|_{L^2(\R^{n+1}_+)}.
\end{align*}
\end{itemize}
\end{rmk}

\begin{proof}
We estimate $\left\| w \right\|_{L^2(B_{2})}$ by $\left\| y_{n+1}^{1-2s} \nabla w \right\|_{L^2(B_{5/2}^+)} + \left\| y_{n+1}^{1-2s} w \right\|_{L^2(B_{5/2}^+)}  $ via the Hardy-trace inequality and then use (\ref{eq:ellreg1}) to bound the gradient term via Cacciopolli's inequality.
\end{proof}

Due to the assumed smoothness of the metric, it follows from Proposition \ref{prop:variablefull} and a bootstrap argument that our solutions are smooth. Assuming the analyticity of all involved coefficients, it is possible to deduce analyticity of the restriction of $w$ to the manifold $(M,g)$.

\begin{cor}[Analyticity]
\label{cor:analytic}
Let $(M,g)$ be a real analytic manifold without boundary, let $w:M \times \R_+ \rightarrow \R$ be a (generalized) eigenfunction of (\ref{eq:CSE}), (\ref{eq:GDN}) with $\lambda=1$. Assume that all coefficient function involved in (\ref{eq:CSE}) are analytic. Then $w(y',0):M \rightarrow \R$ is real analytic and for any $p\in M$ we have the estimate
\begin{equation}
\label{eq:analytic}
|\nabla'^{\alpha} w(p,0) | \leq C C_s^k k!
\end{equation}
where $\alpha= (\alpha_1,...,\alpha_{n},0)$ is a (tangential) multi-index of length $k$.
\end{cor}

\begin{proof}
We prove the corollary by applying Proposition \ref{prop:variablefull} inductively. Combined with Morrey's inequality, this allows to deduce the strong pointwise estimates which we claimed in the corollary. We begin with an iteration of the $L^2$-based bounds. Using Kato's trick \cite{Kato96}, choosing a multi-index $\alpha$ of length $|\alpha|=k$ and setting $m = \floor*{\frac{n}{2}}+1$ , we claim that
\begin{equation}
\label{eq:anaind}
\begin{split}
\left\| \varphi^{k} y_{n+1}^{\frac{1-2s}{2}}\nabla \nabla'^{\alpha} w \right\|_{H^m_{\tan}(B_{1}^+(p))}
\leq C C_s^k k!,
\end{split}
\end{equation}
where $\varphi$ is a radial cut-off function which is supported in $B_{1}^+(p)$ satisfies $0<\varphi(y)<1$ as well as $\varphi(y)=1$ in a neighbourhood of $B_{\frac{1}{2}}^+(p)$. Here $H^m_{\tan}$ denotes the usual Sobolev space with respect to the tangential variables.\\
We prove (\ref{eq:anaind}) inductively. As for a sufficiently large constant $C$ the beginning of the induction is given by (an iteration of) Proposition \ref{prop:variablefull}, it remains to provide an argument for the induction step. 
Using the abbreviation 
\begin{align*}
L_s w = \p_{n+1}y_{n+1}^{1-2s}\p_{n+1}w + y_{n+1}^{1-2s} \D_g w + y_{n+1}^{1-2s} b\cdot \nabla w + y_{n+1}^{1-2s} c w,
\end{align*}
and choosing (tangential) multi-indeces $\alpha$ with $|\alpha|=k$, $\beta$ with $|\beta|=1$, we estimate
\begin{equation}
\label{eq:comm}
\begin{split}
&\left\| \varphi^{k+1} y_{n+1}^{\frac{1-2s}{2}}  \nabla (\nabla')^{\alpha+\beta} w \right\|_{H^m_{\tan}(\R^{n+1}_+)} 
 = \left\| \varphi^{k+1} y_{n+1}^{\frac{1-2s}{2}}  \nabla \nabla'^{\beta} (\nabla')^{\alpha} w \right\|_{H^m_{\tan}(\R^{n+1}_+)}  \\
&\leq  \left\|  y_{n+1}^{\frac{1-2s}{2}} \nabla'^{\beta} \nabla \varphi^{k+1} (\nabla')^{\alpha} w \right\|_{H^m_{\tan}(\R^{n+1}_+)} 
+ \left\|  y_{n+1}^{\frac{1-2s}{2}}[ \nabla'^{\beta} \nabla, \varphi^{k+1}] (\nabla')^{\alpha} w \right\|_{H^m_{\tan}(\R^{n+1}_+)} \\
&\leq \left\|  y_{n+1}^{\frac{2s-1}{2}} L_s \varphi^{k+1} (\nabla')^{\alpha} w \right\|_{H^m_{\tan}(\R^{n+1}_+)} 
 + \left\|  y_{n+1}^{\frac{1-2s}{2}}[ \nabla'^{\beta} \nabla, \varphi^{k+1}] (\nabla')^{\alpha} w \right\|_{H^m_{\tan}(\R^{n+1}_+)} \\
&\leq \left\|  y_{n+1}^{\frac{2s-1}{2}} \varphi^{k+1}  L_s (\nabla')^{\alpha} w \right\|_{H^m_{\tan}(\R^{n+1}_+)} +
\left\|  y_{n+1}^{\frac{2s-1}{2}} [L_s, \varphi^{k+1}] (\nabla')^{\alpha} w \right\|_{H^m_{\tan}(\R^{n+1}_+)} \\
& \quad+ \left\|  y_{n+1}^{\frac{1-2s}{2}}[ \nabla'^{\beta} \nabla, \varphi^{k+1}] (\nabla')^{\alpha} w \right\|_{H^m_{\tan}(\R^{n+1}_+)}.
\end{split}
\end{equation}
Computing the commutator contributions and taking into account the support properties of $\varphi$ and $\varphi'$, we may apply the inductive hypothesis to the last two contributions which yields control over these. Indeed, the first commutator amounts to
\begin{align*}
[\nabla'^{\beta} \nabla ,\varphi^{k+1} ] = & \ (k+1)[\varphi^k(\nabla \varphi)\nabla'^{\beta} + \varphi^k(\nabla'^{\beta} \varphi)\nabla + k\varphi^{k-1}(\nabla'^{\beta}\varphi )(\nabla \varphi) \\
&+ (k+1)\varphi^{k-1} \varphi(\nabla \nabla'^{\beta}\varphi)]
\end{align*}
Thus,
\begin{align*}
&\left\|  y_{n+1}^{\frac{1-2s}{2}}[ \nabla'^{\beta} \nabla, \varphi^{k+1}] (\nabla')^{\alpha} w \right\|_{H^m_{\tan}(\R^{n+1}_+)}\\
&\leq 2(k+1) \left\| \nabla \varphi \right\|_{L^{\infty}} \left\| \varphi^{k}  y_{n+1}^{\frac{1-2s}{2}}  \nabla (\nabla')^{\alpha} w \right\|_{H^m_{\tan}(\R^{n+1}_+)}\\
& \quad + (k+1)k\left\| \nabla \varphi \right\|_{L^{\infty}}^2 \left\|  \varphi^{k-1} y_{n+1}^{\frac{1-2s}{2}} \nabla (\nabla')^{\alpha-\beta} w \right\|_{H^m_{\tan}(\R^{n+1}_+)}\\
& \quad + (k+1)\left\| \varphi \nabla^2 \varphi \right\|_{L^{\infty}}^2 \left\| \varphi^{k-1} y_{n+1}^{\frac{1-2s}{2}} \nabla (\nabla')^{\alpha-\beta} w \right\|_{H^m_{\tan}(\R^{n+1}_+)}\\
&\leq C(k+1)! C_{s}^k.
\end{align*}
We argue similarly for the other commutator contribution. Here we take into account that, due to the radial dependence of $\varphi$,
\begin{align*}
|y_{n+1}^{-2s} \p_{n+1}\varphi| \lesssim |y_{n+1}^{1-2s}\varphi' |y|^{-1}| \lesssim |y_{n+1}^{1-2s}\varphi'| .
\end{align*}
Therefore it remains to deal with the first contribution in (\ref{eq:comm}). We note that $\nabla'^{\alpha }w$ satisfies 
\begin{align*}
\p_{n+1} y_{n+1}^{1-2s} \p_{n+1} \nabla'^{\alpha} w + y_{n+1}^{1-2s} \D_g \nabla'^{\alpha} w &= \sum\limits_{\beta \leq \alpha} {\alpha \choose \beta} y_{n+1}^{1-2s} \nabla'^{\beta} f(g) \nabla'^{\alpha- \beta} u \\
&+ l.o.t.,\\
-\lim\limits_{y_{n+1}\rightarrow 0} y_{n+1}^{1-2s} \p_{n+1} \nabla'^{\alpha} w &= \nabla'^{\alpha} w.
\end{align*}
where $\beta$, $\alpha - \beta$  are a (tangential) multi-indeces and $f$ denotes an analytic function of the metric $g^{ij}$. Here we only explicitly deal with the lower order terms originating from the metric, the terms which originate from the lower order contributions in (\ref{eq:CSE}) can be dealt with analogously.\\
Due to the analyticity of $g$, we have
\begin{align*}
|\varphi^\beta \nabla'^{\beta} f(g)| \leq C |\beta|! R^{|\beta|}
\end{align*}
for some $R\in \R$.
Thus, using Proposition \ref{prop:variablefull} and the explicit form of the right hand side, we deduce 
\begin{align*}
\left\| y_{n+1}^{\frac{2s-1}{2}} \varphi^{k+1} L_s(\nabla')^{\alpha}w \right\|_{H^m_{\tan}(\R^{n+1}_+)}  
&\leq  C\sum\limits_{\beta \leq \alpha}{\alpha \choose \beta } |\beta|!R^{|\beta|}  \left\|y_{n+1}^{1-2s} \varphi^{|\alpha - \beta|} \nabla'^{\alpha-\beta}w\right\|_{H^m_{\tan}(\R^{n+1}_+)} \\
&\leq  C\sum\limits_{\beta \leq \alpha}{\alpha \choose \beta } |\beta|!R^{|\beta|} C|\alpha- \beta|! C_s^{|\alpha- \beta|}\\
&\leq C C_s^{|\alpha|} |\alpha|!.
\end{align*}
Here we first used the equation satisfied by $\nabla'^{\alpha} w$ and then applied the inductive hypothesis assuming that the constants $C_s, C$ in (\ref{eq:anaind}) are chosen sufficiently large. This implies the claim.\\
Using (\ref{eq:anaind}) in combination with Morrey's inequality then permits to deduce pointwise analytic bounds.
\end{proof}

\subsection*{Acknowledgments}
I would like to thank Herbert Koch, Yannick Sire, Stefan Steinerberger and Christian Zillinger for stimulating discussions related to this article. Moreover, I gratefully acknowledge funding from the European Research Council under the European Union’s Seventh Framework Programme (FP7/2007-2013) / ERC grant agreement no. 291053.

\bibliography{citations}

\newcommand{\etalchar}[1]{$^{#1}$}
\begin{thebibliography}{OLBC10}

\bibitem[Bak11]{Bakri}
Laurent Bakri.
\newblock Quantitative uniqueness for {S}chr{\"o}dinger operators.
\newblock {\em arXiv preprint arXiv:1105.5247}, 2011.

\bibitem[BL14]{BL14}
Katarina Bellova and Fanghua Lin.
\newblock Nodal {S}ets of {S}teklov {E}igenfunctions.
\newblock {\em arXiv preprint arXiv:1403.0647}, 2014.

\bibitem[Br{\"u}78]{B}
Jochen Br{\"u}ning.
\newblock {\"U}ber {K}noten von {E}igenfunktionen des
  {L}aplace-{B}eltrami-{O}perators.
\newblock {\em Mathematische Zeitschrift}, 158(1):15--21, 1978.

\bibitem[CK10]{CK}
Ferruccio Colombini and Herbert Koch.
\newblock Strong unique continuation for products of elliptic operators of
  second order.
\newblock {\em Trans. Amer. Math. Soc}, 362(1):345--355, 2010.

\bibitem[CM11]{CM11}
Tobias~H. Colding and William~P. Minicozzi.
\newblock Lower bounds for nodal sets of eigenfunctions.
\newblock {\em Communications in Mathematical Physics}, 306(3):777--784, 2011.

\bibitem[CS07]{CaS}
Luis Caffarelli and Luis Silvestre.
\newblock An {E}xtension {P}roblem {R}elated to the {F}ractional {L}aplacian.
\newblock {\em Communications in Partial Differential Equations},
  32(8):1245--1260, 2007.

\bibitem[Dav13]{Da}
Blair Davey.
\newblock Some {Q}uantitative {U}nique {C}ontinuation {R}esults for
  {E}igenfunctions of the {M}agnetic {S}chr{\"o}dinger {O}perator.
\newblock {\em Communications in Partial Differential Equations},
  (just-accepted), 2013.

\bibitem[DF88]{DoF}
Harold Donnelly and Charles Fefferman.
\newblock Nodal sets of eigenfunctions on {R}iemannian manifolds.
\newblock {\em Inventiones mathematicae}, 93(1):161--183, 1988.

\bibitem[EA97]{E}
L.~Escauriaza and V.~Adolfsson.
\newblock ${C}^{1,\alpha}$ domains and unique continuation at the boundary.
\newblock {\em Comm. Pure Appl. Math., L}, pages 935--969, 1997.

\bibitem[FF13]{FF}
Mouhamed Fall and Veronica Felli.
\newblock Unique continuation property and local asymptotics of solutions to
  fractional elliptic equations.
\newblock {\em ArXiv preprint, arXiv:1301.5119}, 2013.

\bibitem[GL87]{GL}
Nicola Garofalo and Fang-Hua Lin.
\newblock Unique continuation for elliptic operators: {A} geometric-variational
  approach.
\newblock {\em Communications on pure and applied mathematics}, 40(3):347--366,
  1987.

\bibitem[Her77]{He}
Ira~W. Herbst.
\newblock Spectral theory of the operator {$(p^2+ m^2)^{ 1/2}- Ze^2/r$}.
\newblock {\em Communications in Mathematical Physics}, 53(3):285--294, 1977.

\bibitem[HL10]{HL10}
Qing Han and Fang-Hua Lin.
\newblock Nodal sets of solutions of elliptic differential equations.
\newblock {\em Books available on Han’s homepage}, 2010.

\bibitem[HS89]{HS}
Robert Hardt and Leon Simon.
\newblock Nodal sets for solutions of elliptic equations.
\newblock {\em Journal of differential geometry}, 30(2):505--522, 1989.

\bibitem[HS11]{HS11}
Hamid Hezari and Christopher~D. Sogge.
\newblock A natural lower bound for the size of nodal sets.
\newblock {\em arXiv preprint arXiv:1107.3440}, 2011.

\bibitem[HW12]{HW12}
Hamid Hezari and Zuoqin Wang.
\newblock Lower bounds for volumes of nodal sets: an improvement of a result of
  {S}ogge-{Z}elditch.
\newblock {\em Spectral Geometry}, 84:229, 2012.

\bibitem[JK85]{JK}
David Jerison and Carlos~E. Kenig.
\newblock Unique continuation and absence of positive eigenvalues for
  {S}chr{\"o}dinger operators.
\newblock {\em The Annals of Mathematics}, 121(3):463--488, 1985.

\bibitem[K{\etalchar{+}}98]{KukaI}
Igor Kukavica et~al.
\newblock Quantitative uniqueness for second-order elliptic operators.
\newblock {\em Duke mathematical journal}, 91(2):225--240, 1998.

\bibitem[Kat96]{Kato96}
Keiichi Kato.
\newblock New idea for proof of analyticity of solutions to analytic nonlinear
  elliptic equations.
\newblock {\em SUT J. Math}, 32(2):157--161, 1996.

\bibitem[KEA95]{KEA}
Carlos~E. Kenig, L.~Escauriaza, and V.~Adolfsson.
\newblock Convex domains and unique continuation at the boundary.
\newblock {\em Revista Matem{\'a}tica Iberoamericana}, 11(3):513--526, 1995.

\bibitem[KT01]{KT1}
Herbert Koch and Daniel Tataru.
\newblock Carleman estimates and unique continuation for second-order elliptic
  equations with nonsmooth coefficients.
\newblock {\em Communications on Pure and Applied Mathematics}, 54(3):339--360,
  2001.

\bibitem[KW98]{KW}
Carlos~E. Kenig and Wensheng Wang.
\newblock A {N}ote on {B}oundary {U}nique {C}ontinuation for {H}armonic
  {F}unctions in {N}on-{S}mooth {D}omains.
\newblock {\em Potential Analysis}, 8(2):143--147, 1998.

\bibitem[LM73]{LM}
Jacques-Louis Lions and Enrico Magenes.
\newblock {\em Non-homogeneous boundary value problems and applications}.
\newblock Springer, 1973.

\bibitem[LR95]{LRob}
Gilles Lebeau and Lorenzo Robbiano.
\newblock Contr{\'o}le exact de l'{\'e}quation de la chaleur.
\newblock {\em Communications in Partial Differential Equations},
  20(1-2):335--356, 1995.

\bibitem[LZ98]{LZ98}
Gilles Lebeau and Enrique Zuazua.
\newblock Null-{C}ontrollability of a {S}ystem of {L}inear {T}hermoelasticity.
\newblock {\em Archive for rational mechanics and analysis}, 141(4):297--329,
  1998.

\bibitem[OLBC10]{OLBC}
Frank Olver, Daniel Lozier, Ronald Boisvert, and Charles Clark.
\newblock {\em {N}{I}{S}{T} {H}andbook of {M}athematical {F}unctions}.
\newblock Cambridge University Press and the National Institute of Standards
  and Technology, 2010.

\bibitem[R{\"u}l14a]{Rue}
Angkana R{\"u}land.
\newblock On {S}ome {R}igidity {P}roperties in {P}{D}{E}s.
\newblock {\em Dissertation, University of Bonn}, 2014.

\bibitem[R{\"u}l14b]{Rue14}
Angkana R{\"u}land.
\newblock Unique {C}ontinuation for {F}ractional {S}chr{\"o}dinger {E}quations
  with {R}ough {P}otentials.
\newblock {\em Communications in Partial Differential Equations},
  (just-accepted), 2014.

\bibitem[Seo13a]{Seo}
Ihyeok Seo.
\newblock {O}n {U}nique {C}ontinuation for {S}chr{\"o}dinger {O}perators of
  {F}ractional and {H}igher {O}rders.
\newblock {\em ArXiv preprint, arXiv:1301.2460v1}, 2013.

\bibitem[Seo13b]{Seo1}
Ihyeok Seo.
\newblock Unique {C}ontinuation for {F}ractional {S}chr{\"o}dinger {O}perators
  in {T}hree and {H}igher {D}imensions.
\newblock {\em ArXiv preprint, arXiv:1309.0120v1}, 2013.

\bibitem[Ste13]{S13}
Stefan Steinerberger.
\newblock Lower bounds on nodal sets of eigenfunctions via the heat flow.
\newblock {\em arXiv preprint arXiv:1301.3371}, 2013.

\bibitem[SZ11]{SZ11}
Christopher~D. Sogge and Steve Zelditch.
\newblock Lower bounds on the {H}ausdorff measure of nodal sets.
\newblock {\em Math. Res. Let.}, (18):25--37, 2011.

\bibitem[Yau82]{Y82}
Shing-Tung Yau.
\newblock {\em Seminar on differential geometry}.
\newblock Number 102. Princeton University Press, 1982.

\bibitem[Zel14]{Zel14}
Steve Zelditch.
\newblock Measure of nodal sets of analytic {S}teklov eigenfunctions.
\newblock {\em arXiv preprint arXiv:1403.0647}, 2014.

\bibitem[Zhu13]{Zhu13}
Jiuyi Zhu.
\newblock Quantitative uniqueness of elliptic equations.
\newblock {\em arXiv preprint arXiv:1312.0576}, 2013.

\bibitem[Zhu14]{Zhu14}
Jiuyi Zhu.
\newblock Doubling property and vanishing order of {S}teklov eigenfunctions.
\newblock {\em personal communication}, 2014.

\end{thebibliography}
\bibliographystyle{alpha}
\end{document}